\documentclass[12pt, reqno]{amsart}
\usepackage{amsmath,amssymb,verbatim,comment,color}
\usepackage[pagebackref,pdftex,hyperindex]{hyperref}
\usepackage[margin=3cm]{geometry}
\usepackage{url}
\usepackage{color}
\newcommand{\C}{\mathbb{C}}

\newcommand{\N}{\mathbb{N}}

\renewcommand{\P}{\mathbb{P}}
\newcommand{\Q}{\mathbb{Q}}
\newcommand{\R}{\mathbb{R}}
\renewcommand{\S}{\mathbb{S}}

\newcommand{\fu}{\mathfrak{u}}

\newcommand{\cD}{\mathcal{D}}

\newcommand{\cF}{\mathcal{F}}

\newcommand{\cH}{\mathcal{H}}

\newcommand{\cJ}{\mathcal{J}}

\newcommand{\cL}{\mathcal{L}}

\newcommand{\cO}{\mathcal{O}}
\newcommand{\cP}{\mathcal{P}}

\newcommand{\cR}{\mathcal{R}}

\newcommand{\cX}{\mathcal{X}}

\renewcommand{\a}{\alpha}
\renewcommand{\b}{\beta}
\newcommand{\g}{\gamma}
\renewcommand{\d}{\delta}
\newcommand{\e}{\varepsilon}

\renewcommand{\r}{\rho}

\newcommand{\bv}{\boldsymbol v}
\newcommand{\bw}{\boldsymbol w}
\newcommand{\bA}{\boldsymbol A}

\newcommand{\bH}{\boldsymbol H}

\newcommand{\bet}{\boldsymbol \eta}
\newcommand{\bmu}{\boldsymbol \mu}
\newcommand{\bnu}{\boldsymbol \nu}
\newcommand{\bphi}{\boldsymbol \phi}

\newcommand{\bomega}{\boldsymbol \omega}

\newcommand{\p}{\partial}
\newcommand{\bp}{\bar{\partial}}
\newcommand{\dd}{\sqrt{-1}\partial \bar{\partial}}
\newcommand{\ddt}{\frac{d}{dt}}
\newcommand{\cf}{{\rm cf.\ }} 
\newcommand{\eg}{{\rm e.g.\ }} 
\newcommand{\ie}{{\rm i.e.\ }} 

\DeclareMathOperator{\AM}{AM}
\DeclareMathOperator{\Aut}{Aut}

\DeclareMathOperator{\Chow}{Chow}
\DeclareMathOperator{\CKE}{CKE}
\DeclareMathOperator{\diag}{diag}
\DeclareMathOperator{\dist}{dist}
\DeclareMathOperator{\DF}{DF}
\DeclareMathOperator{\Fut}{Fut}

\DeclareMathOperator{\FS}{FS}
\DeclareMathOperator{\aFS}{{\rm FS}^{\rm amb}}
\DeclareMathOperator{\bFS}{\tilde{FS}}
\DeclareMathOperator{\GL}{GL}

\DeclareMathOperator{\Hilb}{Hilb}
\DeclareMathOperator{\bHilb}{\tilde{Hilb}}

\DeclareMathOperator{\Id}{Id}

\DeclareMathOperator{\Ker}{Ker}
\DeclareMathOperator{\Lie}{Lie}
\DeclareMathOperator{\MA}{MA}
\DeclareMathOperator{\bMA}{\bf{MA}}
\DeclareMathOperator{\op}{op}

\DeclareMathOperator{\PSH}{PSH}
\DeclareMathOperator{\osc}{osc}

\DeclareMathOperator{\Ric}{Ric}

\DeclareMathOperator{\Tr}{Tr}

\renewcommand{\leq}{\leqslant}
\renewcommand{\geq}{\geqslant}

\renewcommand{\hat}{\widehat}
\renewcommand{\tilde}{\widetilde}
\newcommand{\elm}[1]{\overset{\vee}{#1}}

\numberwithin{equation}{section}       % Number formulas within sections

\newtheorem{prop} {Proposition} [section]
\newtheorem{thm}[prop] {Theorem} 
\newtheorem{dfn}[prop] {Definition}
\newtheorem{lem}[prop] {Lemma}

\newtheorem{rem}[prop]{Remark}
\theoremstyle{remark}
\newtheorem*{ackn}{\bf{Acknowledgment}}

\title[Geometric quantization of coupled K\"ahler-Einstein metrics]{Geometric quantization of \\coupled K\"ahler-Einstein metrics}
\date{\today}

\author[R. Takahashi]{Ryosuke Takahashi}
\address{Research Institute for Mathematical Sciences\\
 Kyoto University\\
 Kyoto 606-8502\\
 JAPAN}
\email{tryosuke@kurims.kyoto-u.ac.jp}

\begin{document}

\maketitle
\begin{abstract}
We study the quantization of coupled K\"ahler-Einstein (CKE) metrics, namely we approximate CKE metrics by means of the canonical Bergman metrics, so called the ``balanced metrics''. We prove the existence and weak convergence of balanced metrics for the negative first Chern class, while for the positive first Chern class, we introduce some algebro-geometric obstruction which interpolates between the Donaldson-Futaki invariant and Chow weight. Then we show the existence and weak convergence of balanced metrics on CKE manifolds under the vanishing of this obstruction. Moreover, restricted to the case when the automorphism group is discrete, we also discuss approximate solutions and a gradient flow method towards the smooth convergence.
\end{abstract}
\tableofcontents
%==============Section 1==========================
\section{Introduction}
Let $X$ be an $n$ dimensional compact K\"ahler manifold with $\lambda c_1(X)>0$ ($\lambda=\pm 1$). In K\"ahler geometry, the problem of finding canonical metrics is one of the most active topics. In particular, the K\"ahler-Einstein condition
\[
\Ric(\omega)=\lambda \omega
\]
has been studied extensively in the past decades. The existence of K\"ahler-Einstein metrics is strongly related to a certain subtle properties in algebraic geometry. Actually, the resolution to Yau-Tian-Donaldson conjecture \cite{CDS15a, CDS15b, CDS15c} states that a Fano manifold admits a K\"ahler-Einstein metric if and only if it is K-polystable, an algebro-geometric notion modeled on Geometric Invariant Theory (GIT). The resolution of this conjecture enables us to construct new examples of K\"ahler-Einstein metrics \cite{Del16, Sus13} and their good moduli spaces \cite{Oda15}.

However, there are many Fano manifolds which do not satisfy the stability condition. For this reason, several generalizations of K\"ahler-Einstein metrics have been studied. In this paper, we will focus on the ``coupled K\"ahler-Einstein metrics'' introduced  by Hultgren-Witt Nystr\"om \cite{HN17}: for any $N \in \N$, an $N$-tuple of K\"ahler classes $(\Omega_{[i]})$ is called a decomposition of $2\pi \lambda c_1(X)$ if
\[
2\pi \lambda c_1(X)=\sum_{i=1}^N \Omega_{[i]}.
\]
Then an $N$-tuple of K\"ahler metrics $\bomega=(\omega_{[i]}) \in \prod_{i=1}^N \Omega_{[i]}$ is called {\it coupled K\"ahler-Einstein} (CKE) if it satisfies
\[
\Ric(\omega_{[i]})=\lambda \sum_{j=1}^N \omega_{[j]}, \quad i=1,\ldots, N.
\]
Note that if $N=1$, then $\omega_{[1]}$ is CKE if and only if $\omega_{[1]}$ is K\"ahler-Einstein. Hultgren \cite{Hul17} studied an example of Fano $4$-folds discovered by Futaki \cite{Fut83} which does not admit any K\"ahler-Einstein metrics, and showed that in a special case of the decomposition for $2\pi \lambda c_1(X)$, there exists a CKE metrics on this manifold. Thereafter this example is further studied in \cite{FZ19}. CKE metrics has been studied extensively in recent years \cite{Del19, DH18, DP19, FZ18,  FZ19, Hul17, Pin18, Tak19}. One of the motivation for CKE metrics comes from stabilities. Indeed, Datar-Pingali \cite{DP19} recently showed that CKE metrics have an infinite dimensional moment map/GIT interpretation which interpolates between polarized manifolds and holomorphic vector bundles over them. This result strongly indicates a potential application to construct their good moduli space. Such kind of interpolation also can be found in other frameworks \cite{CFP13, CY18}.

The main issue of this paper is to consider the quantization of CKE metrics in the sense of Donaldson \cite{Don01}: now assume that we have a {\it rational decomposition} of $K_X^{\otimes -\lambda}$, \ie there exist an $N$-tuple of $\Q$-ample line bundles $(L_{[i]})$ over $X$ such that
\[
K_X^{-\otimes \lambda} \simeq \otimes_{i=1}^N L_{[i]}.
\]
In particular, this gives a decomposition of $2\pi \lambda c_1(X)$ by setting $\Omega_{[i]}:=2\pi c_1(L_{[i]})$.

We establish a precise correspondence between CKE metrics and the large $k$ asymptotics of the sequences of $N$ projective embeddings:
\[
X^N \hookrightarrow \P^\ast(H^0(X,L_{[1]}^{\otimes k})) \times \cdots \times \P^\ast(H^0(X,L_{[N]}^{\otimes k})).
\] 
In our framework, a {\it balanced metric} is an $N$-tuple of Bergman metrics governed by the system of equations which reflects their mutual interaction, more precisely, arises as a unique fixed point modulo automorphisms, of the {\it balancing flow}:
\begin{equation} \label{balflow}
\begin{cases}
\ddt H_{[1]}(t)=k(\Id-D_{[1]}^{(k)} \overline{M}_{[1], \bmu}(\bH(t))) \\
\hspace{20mm} \vdots \\
\ddt H_{[N]}(t)=k(\Id-D_{[N]}^{(k)} \overline{M}_{[N], \bmu}(\bH(t)))
\end{cases}
\end{equation}
for $N$-tuple of Hermitian matrices $\bH(t)=(H_{[i]}(t))$. Where $D_{[i]}^{(k)}$ denotes the dimension of $H^0(X, L_{[i]}^{\otimes k})$ and the matrix $\overline{M}_{[i], \bmu}(\bH(t)))$ is defined by the integration of the moment map for the standard $H_{[i]}(t)$-unitary action on $\P^\ast(H^0(X, L_{[i]}^{\otimes k}))$. The definition of \eqref{balflow} is motivated by another flow on the space of K\"ahler potentials as follows: by using the Berezin-Toeplitz quantization as in \cite{Fin10, Tak18}, one can easily see that the balancing flow \eqref{balflow} is the quantization of the {\it coupled inverse Monge-Amp\`ere flow}:
\begin{equation} \label{imf}
\begin{cases}
\ddt \phi_{[1]}(t)=1-e^{\rho_{[1]}(t)} \\
\hspace{20mm} \vdots \\
\ddt \phi_{[N]}(t)=1-e^{\rho_{[N]}(t)},
\end{cases}
\end{equation}
where $\rho_{[i]}(t)$ denotes the coupled Ricci potential for the evolving metric $\phi_{[i]}(t)$ defined by
\[
\Ric(\omega_{\phi_{[i]}(t)})-\lambda \sum_j \omega_{\phi_{[j]}(t)}=\dd \rho_{[i]}(t), \quad \int_X e^{\rho_{[i]}(t)} \omega_{\phi_{[i]}(t)}^n=(2\pi L_{[i]})^n.
\]
The coupled flow \eqref{imf} is clearly parabolic and short-time existence immediately follows form general theory. Moreover, when $N=1$, the long-time existence as well as some convergence results of \eqref{imf} have been established very recently by Collins, Hisamoto and the author \cite{CHT17}. We hope that the flow \eqref{imf} also exists for all positive time and provides a powerful way to construct CKE metrics.
%==============Subsection 1.1==========================
\subsection{Main results}
It is known that there always exists a unique CKE metric in the $\lambda=-1$ case \cite{HN17}. Correspondingly, we can prove the existence of balanced metrics as well (cf. Theorem \ref{crconv}). On the other hand, in the $\lambda=1$ case, there is some holomorphic obstruction to the existence of balanced metrics which interpolates between the coupled Futaki invariant $\Fut_c$ for the $N$-tuple $(X, (L_{[i]}))$ (\cf \cite{DP19}) and the higher order Futaki invariants $\cF_{[i],p}$ $(p=1,\ldots, n$) for each polarization $(X, L_{[i]})$ (\cf \cite{Fut04}). Indeed, our new invariant, referred as the {\it higher order coupled Futaki invariants} $\cF_{c, j}$ ($j=1,\ldots, nN+1$) can be expressed as linear combinations of $\Fut_c$ and $\cF_{[i],p}$ ($p=1,\ldots, n$), whose coefficients comes from those of the Hilbert polynomials of $\{(X, L_{[i]})\}_{i=1,\ldots, N}$. Conversely, if this obstruction vanishes, we can show the following:
\begin{thm} \label{balp}
Let $X$ be a compact K\"ahler manifold with $c_1(X)>0$ and $(L_{[i]})$ a rational decomposition of $K_X^{\otimes -1}$. Assume $(X,(L_{[i]}))$ admits a coupled K\"ahler-Einstein metric and all the higher order coupled Futaki invariants vanish. Then for any sufficiently large and divisible $k$, there exists a balanced metric $\bH^{(k)} \in \prod_i \cH_{[i]}^{(k)}$ which is unique modulo automorphisms, and the corresponding Bergman metric $\bFS^{(k)}(\bH^{(k)})$ converges modulo automorphisms to a coupled K\"ahler-Einstein metric in the weak topology.
\end{thm}

According to \cite{Don01}, we also study approximate solutions, the so called {\it almost balanced metrics} $\bphi_\ell$ for arbitrary fixed positive integer $\ell$, defined as a formal power series. The construction of almost balanced metrics heavily depends on the asymptotic expansion of (normalized) Bergman kernels $\bar{B}_{[i],\bmu}^{(k)}$ ($i=1,\ldots, N$). With this notion, our balanced condition can be written as
\[
\bar{B}_{[1],\bmu}^{(k)}=\ldots=\bar{B}_{[N],\bmu}^{(k)}=1,
\]
while we can construct almost balanced metrics $\bphi_\ell$ satisfying the following property:
\begin{thm} \label{happroximation}
Let $X$ be a compact K\"ahler manifold with $\lambda c_1(X)>0$ ($\lambda=\pm 1$), and $(L_{[i]})$ a rational decomposition of $K_X^{\otimes -\lambda}$. In the $\lambda=1$ case, we further assume that the decomposition $(L_{[i]})$ admits a coupled K\"ahler-Einstein metric and $\Aut_0(X)$ is trivial. Let $\bphi_{\CKE} \in \prod_i \cH_{[i]}$ denote the unique coupled K\"ahler-Einstein metric. Then there exist $\bet_1, \ldots, \bet_\ell \in (C^\infty(X;\R))^N$ such that $\bphi_\ell:=\bphi_{\CKE}+\sum_{j=1}^\ell k^{-j} \bet_j$ satisfies
\begin{equation} \label{hberg}
\bar{B}_{[i], \bmu}^{(k)}(\bphi_\ell)=1+O(k^{-\ell-1}), \quad i=1,\ldots, N.
\end{equation}
\end{thm}
The almost balanced metrics play a crucial role in the study of the $C^\infty$-convergence of balanced metrics. Roughly speaking, the almost balanced condition implies that the corresponding Bergman metric has a small initial gradient vector along the balancing flow. This geometric interpretation helps us to obtain convergence of the flow and the following:
\begin{thm} \label{crconv}
Under the same assumptions as in Theorem \ref{happroximation}, the balanced metric $\bH^{(k)} \in \prod_i \cH_{[i]}^{(k)}$ exists for sufficiently large and divisible $k$, and the corresponding K\"ahler metric $\bomega_{\bFS^{(k)} (\bH^{(k)})}$ converges to a coupled K\"ahler-Einstein metric in the $C^\infty$-topology.
\end{thm}
The core of the paper is Section \ref{Gquant}, where we define the balanced metrics in the coupled setting and study the basic property of them. Theorem \ref{balp} are proved by a variational method, \ie establishing some uniform coercivity estimates for (quantized) Ding functionals. This can be done by a combination of \cite{Tak19} and a simple extension of the arguments for the $N=1$ case (\cf \cite{BBGZ13, BN14}). The proofs of Theorem \ref{happroximation} and Theorem \ref{crconv} also proceed along the same line as \cite{Don01}. However we emphasize that the proof of Theorem \ref{crconv} does not rely on Theorem \ref{balp}, and $C^\infty$-convergence result is new even in the $N=1$ case as described below.
\subsection{Relation to other results}
When $N=1$, Donaldson \cite{Don01} studied canonical Bergman metrics (also called balanced metrics in his literature) defined as the fixed points of the map $T_{\MA}^{(k)} \colon \cH^{(k)} \rightarrow \cH^{(k)}$ (see Section \ref{sfcc} for more details). Such a metric can be considered for arbitrary polarization $L \to X$ and quantizes the constant scalar curvature K\"ahler metrics in the $C^\infty$-topology. The proof also relied on the balancing flow method, which was further exploited in several settings (\eg \cite{Fin10, Has15}). We remark that the notion of Donaldson's balanced metric does not agree with that introduced by the author. So in the K\"ahler-Einstein settings, we have two $C^\infty$-quantization schemes. It would be interesting to compare the asymptotic behavior of these two balanced metrics.

On the other hand, the notion of our balanced metrics \eqref{bmetric} has already been studied in several papers \cite{BBGZ13, BN14, Tak15} in the $N=1$ case, where the existence and weak convergence results of the balanced metrics were given under some suitable assumptions.
%==============Subsection 1.2==========================
\subsection{Organization of the paper}
This paper is organized as follows. In Section \ref{CDF}, we review the definitions of several functionals on the space of K\"ahler potentials and coercivity of the Ding functional. Then we study some properties for the weighted Laplacian (\cf Proposition \ref{wlaplacian}). In Section \ref{Gquant}, we introduce the quantized space, functionals in the coupled setting. All of these notions are introduced by a combination of those in the case $N=1$ which are well-known for experts (for instance, see \cite{BBGZ13}). We study some variational properties and asymptotic slope of the functionals along Bergman geodesic rays. In particular, applying to the Bergman geodesic rays which originate from test configurations, we obtain an algebro-geometric obstruction. In Section \ref{ewcbm}, we prove the existence and weak convergence of balanced metrics by using a variational method. In Section \ref{TCR}, we construct almost balanced metrics by solving linear systems \eqref{cleq} which naturally arise from the asymptotic expansion of Bergman kernels. Then we review the definition of $R$-bounded geometry and some related results. Finally, we study the properties for the balancing flow starting from almost balanced metrics and prove Theorem \ref{crconv}.
%==============Subsection 1.3==========================
\subsection{Conventions and remarks}
Throughout the paper, we assume that each $L_{[i]}$ is a line bundle to simplify the expressions. However, we remark that in general case, all the results in this paper still hold by considering only {\it divisible} $k$.

We make the convention that $X$ is a compact K\"ahler manifold with a reference K\"ahler form $\hat{\omega}$ in a K\"ahler class $\Omega$. When the class $\Omega$ is represented by some ($\Q$-) ample line bundle, we also fix a fiber metric $\hat{h}$ so that the curvature of $\hat{h}$ coincides with $\hat{\omega}$.

For the coupled settings, we use the bold symbol (\eg $\bphi$, $\bnu$), which denotes an $N$-tuple of functions, K\"ahler forms, measures, and so on, whereas the calligraphy or tilde like $\cJ$, $\cD$, $\tilde{S}$ are often used for functionals and maps. Whenever no further comment is made, the $i$-th component $(i=1,\ldots, N)$ is expressed by the special index ``$[i]$'' in order to distinguish from other indices. The notation $[i]$ is also used to denote several notions just obtained by applying the single setup to the $i$-th component $(X, \Omega_{[i]})$. 

For the quantized settings, the index ``$(k)$'' means the level of quantization, \ie the exponent of $L^{\otimes k}$. 

For the Fubini-Study metrics, we distinguish the K\"ahler form $\omega_{\FS^{(k)}(H)}$ associated to the image of the Fubini-Study map from the Fubini-Study metric $\omega_{\aFS(H)} \in 2 \pi c_1(H^0(X,L^{\otimes k}), \cO(1))$ on the ambient projective space. Restricted to $X$, the relation between these metrics is given by $\omega_{\FS^{(k)}(H)}=k^{-1} \omega_{\aFS(H)}$.

For the $G$-linearizations in the case when $\lambda=1$ and $G$ is non-trivial, we know that $G$ is a linear algebraic group (since $X$ is Fano), and hence there exists an integer $k_0$ such that each $L_{[i]}^{\otimes k_0}$ has a $G$-linearization (\cf \cite[Chapter 7]{Dol03}). Moreover, since the $G$-linearization is unique up to an overall constant multiple, adjusting constants, we obtain a $G$-equivariant isomorphism
\[
K_X^{\otimes -k_0} \simeq L_{[1]}^{\otimes k_0} \otimes \cdots \otimes L_{[N]}^{\otimes k_0}.
\]
In later arguments we assume that $k_0=1$ for simplicity (again, in general case, all the results in this paper still hold by considering only divisible $k$).
%==============Acknowledgement==========================
\begin{ackn}
The author expresses his gratitude to Prof. Tomoyuki Hisamoto, Ryoichi Kobayashi and Dr. Satoshi Nakamura for reading the draft of the paper. Also the author is grateful to the referees for many insightful comments which have helped to improve the article.
\end{ackn}
%==============Section 2==========================
\section{Preliminaries} \label{CDF}
%==============Subsection 2.1==========================
\subsection{Coercivity of the Ding functional}
Let $X$ be an $n$-dimensional compact K\"ahler manifold with $\lambda c_1(X)>0$, $\hat{\omega}$ a K\"ahler metric in a K\"ahler class $\Omega$. First, we set
\[
\cH:=\{ \phi \in C^{\infty}(X;\R)|\omega_{\phi}:=\hat{\omega}+\dd \phi>0\}.
\]
The {\it Aubin-Mabuchi energy} is defined by
\[
\AM(\phi):=\frac{1}{(n+1)\Omega^n} \sum_{j=0}^n \int_X \phi \omega_{\phi}^j \wedge \hat{\omega}^{n-j}.
\]
The Aubin-Mabuchi energy satisfies the scaling property $\AM(\phi+c)=\AM(\phi)+\Omega^n \cdot c$ for all $\phi \in \cH$ and $c \in \R$. For $\phi \in \cH$, we set
\[
\MA(\phi):=\frac{\omega_\phi^n}{\Omega^n}
\]
so that $\MA(\phi)$ becomes the propability measure on $X$. We define the functionals $J$, $\hat{L}$ on $\cH$ as follows
\[
\hat{L}(\phi):=\int_X \phi \MA(0),
\]
\[
J(\phi)=J(\omega_\phi):=-\AM(\phi)+\hat{L}(\phi),
\]
Then the variational formula for $J$ is given by
\[
\d J|_\phi(\d \psi)=-\int_X \d \psi(\MA(\phi)-\MA(0)).
\]
In particular, this implies that the infimum of $J$ is zero, and attained by $\hat{\omega}$.

Next we work in the CKE setting. Let $(\Omega_{[i]})$ be a decomposition of $2 \pi \lambda c_1(X)$ and $\hat{\bomega}=(\hat{\omega}_{[i]})\in \prod_i \Omega_{[i]}$ an $N$-tuple of K\"ahler metrics. We take also a K\"ahler form $\theta_0$ satisfying the Calabi-Yau equation \cite{Yau78}
\[
\Ric(\theta_0)=\lambda \sum_i \hat{\omega}_{[i]}, \quad \int_X \theta_0^n=1.
\]
For $\bphi \in \prod_i \cH_{[i]}$, we define
\[
\hat{\cL}(\bphi):=\sum_i \hat{L}_{[i]} (\phi_{[i]}),
\]
\[
\cJ(\bphi)=\cJ(\bomega_{\bphi}):=\sum_i J_{[i]}(\phi_{[i]}),
\]
\[
\cL(\bphi):=-\lambda \log \int_X e^{-\lambda \sum_i \phi_{[i]}} \theta_0^n.
\]
Moreover, following \cite{HN17}, the {\it Ding functional} is defined by
\[
\cD(\bphi)=\cD(\bomega_{\bphi}):=-\sum_i \AM_{[i]}(\phi_{[i]})+\cL(\bphi).
\]
It is known that the Ding functional $\cD$ is convex along weak geodesics (\cf \cite[Lemma 3.1, Theorem 4.3]{HN17}). Moreover, a direct computation shows that
\begin{equation} \label{vfding}
\d \cD|_{\bphi}(\d \bphi)=-\sum_i \int_X \d \phi_{[i]} (1-e^{\r_{[i]}(\bphi)})\MA(\phi_{[i]}).
\end{equation}
This yields that a potential $\bphi \in \prod_i \cH_{[i]}$ is CKE if and only if it is a critical point of $\cD$. 

Finally, we take account of the action of automorphisms $\Aut(X)$. Set $G:=\Aut_0(X)$, the identity component of the automorphism group $\Aut(X)$. Then the $G$-invariant functional $\cJ_G \colon \prod_i \cH_{[i]} \to \R$ is defined as
\[
\cJ_G(\bphi)=\cJ_G(\bomega_{\bphi}):=\inf_{f \in G} \cJ(f^\ast \bomega_{\bphi}).
\]
\begin{dfn}
We say that the Ding functional $\cD$ is $\cJ_G$-coercive if there exists some constants $\d, C>0$ such that
\[
\cD(\bphi) \geq \d \cJ_G(\bphi)-C, \quad \bphi \in \prod_i \cH_{[i]}.
\]
\end{dfn}
Then by using the general criterion for coercivity developed by Darvas-Rubinstein \cite{DR15}, the author proved the following:
\begin{thm}[Theorem 1.2, \cite{Tak19}] \label{Coercivity}
Let $X$ be a compact K\"ahler manifold with $c_1(X)>0$ and $(\Omega_{[i]})$ be a decomposition of $2 \pi c_1(X)$. Then $(X, (\Omega_{[i]}))$ admits a coupled K\"ahler-Einstein metric if and only if the Ding functional $\cD$ is $\cJ_G$-coercive.
\end{thm}
On the other hand, we note that in the $\lambda=-1$ case, the $\cJ$-coercivity estimate for $\cD$ holds without any assumptions, which was essentially shown in the course of the proof for \cite[Theorem A]{HN17}.
%==============Subsection 2.2==========================
\subsection{The weighted Laplacian}
For $\bphi \in \prod_i \cH_{[i]}$, we define the {\it canonical measure} $\mu$ by
\begin{equation} \label{canmeasure}
\mu(\bphi)=\mu(\bomega_{\bphi}):=\frac{e^{-\lambda \sum_i \phi_{[i]}} \theta_0^n}{\int_X e^{-\lambda \sum_i \phi_{[i]}} \theta_0^n}.
\end{equation}
Also we define the {\it coupled Ricci potential} $\r_{[i]}=\r_{[i]}(\bphi) \in C^\infty(X;\R)$ by
\[
\Ric(\omega_{\phi_{[i]}(\bphi)})-\lambda \sum_j \omega_{\phi_{[j]}}=\dd \r_{[i]}(\bphi), \quad \int_X e^{\r_{[i]}(\bphi)}\omega_{\phi_{[i]}}^n=\Omega_{[i]}^n.
\]
By the definition, we obtain
\begin{equation} \label{eqmeas}
e^{\r_{[1]}(\bphi)} \MA(\phi_{[1]})=\ldots= e^{\r_{[N]}(\bphi)} \MA(\phi_{[N]})=\mu(\bphi).
\end{equation}
Moreover, $\bphi \in \prod_i \cH_{[i]}$ is CKE if and only if
\[
\r_{[1]}(\bphi)=\ldots=\r_{[N]}(\bphi)=0.
\]
\begin{prop} \label{vfmu}
The variation of $\mu(\bphi)$ and $\r_{[i]}(\bphi)$ along the direction $\d \bphi=(\d \phi_{[1]}, \ldots, \d \phi_{[N]})$ are given by
\[
\d \mu(\d \bphi)=-\lambda \sum_i \d \phi_{[i]} \mu(\bphi)+\lambda \sum_i \int_X  \d \phi_{[i]} \mu(\bphi),
\]
\[
\d \r_{[i]}(\d \bphi)=-\Delta_{\phi_{[i]}} \d \phi_{[i]}-\lambda \sum_i \d \phi_{[i]} \mu(\bphi)+\lambda \sum_i \int_X  \d \phi_{[i]} \mu(\bphi).
\]
\end{prop}
\begin{proof}
A straightforward computation shows that
\begin{eqnarray*}
\d \mu(\d \bphi) &=& \frac{-\lambda \sum_i \d \phi_{[i]} \cdot e^{-\lambda \sum_i \phi_{[i]}} \theta_0^n \cdot \int_X e^{-\lambda \sum_i \phi_{[i]}} \theta_0^n+e^{-\lambda \sum_i \phi_{[i]}} \theta_0^n \cdot \int_X \lambda \sum_i \d \phi_{[i]} e^{-\lambda \sum_i \phi_{[i]}} \theta_0^n}{\bigg( \int_X e^{-\lambda \sum_i \phi_{[i]}} \theta_0^n \bigg)^2}\\
&=& -\lambda \sum_i \d \phi_{[i]} \mu(\bphi)+\lambda \sum_i \int_X  \d \phi_{[i]} \mu(\bphi).
\end{eqnarray*}
Also differentiating the equation $e^{\r_{[i]}(\bphi)} \MA(\phi_{[i]})=\mu(\bphi)$ we get
\[
\d \r_{[i]}(\d \bphi) \mu(\bphi)+\Delta_{\phi_{[i]}} \d \phi_{[i]} \mu(\bphi)=\d \mu (\d \bphi).
\]
Combining with the variational formula of $\mu(\bphi)$, we obtain the desired result.
\end{proof}
For the time being, we denote several quantities with respect to $\omega_{\phi_{[i]}}$ (\eg the Riemannian metric, covariant derivative, Laplacian) by $g_{[i]}$, $\nabla_{[i]}$, $\Delta_{[i]}$, and so on. We define the operator
\[
\cP_{\bphi} \colon (C^\infty(X;\C))^N \longrightarrow (C^\infty(X;\C))^N
\]
by $N$-tuple of weighted Laplacians
\[
\cP_{\bphi}(\bv):=
\begin{pmatrix}
\Delta_{[1]} v_{[1]}+(\p \r_{[1]}, \p \overline{v_{[1]}})_{[1]}+\lambda \sum_j v_{[j]}-\lambda \sum_j \int_X v_{[j]}\mu(\bphi)\\
\vdots \\
\Delta_{[N]} v_{[N]}+(\p \r_{[N]}, \p \overline{v_{[N]}})_{[N]}+\lambda \sum_j v_{[j]}-\lambda \sum_j \int_X v_{[j]}\mu(\bphi)
\end{pmatrix}.
\]
Meanwhile, the space of $N$ functions $(C^\infty(X;\C))^N$ is equipped with a Hermitian product given by
\begin{equation} \label{lproduct}
(\bv, \bw)_{\bmu(\bphi)}:=\sum_i \int_X v_{[i]} \cdot \overline{w_{[i]}} \mu(\bphi), \quad \bv, \bw \in (C^\infty(X;\C))^N.
\end{equation}
The operator $\cP_{\bphi}$ naturally arises as the Hessian of the Ding functional $\cD$. Indeed, for any smooth curves $\bphi:=\bphi(t)$, a direct computation shows that
\begin{eqnarray*}
\frac{d^2}{dt^2} \cD(\bphi) &=& -\sum_i \bigg[ \int_X \ddot{\phi}_{[i]} (1-e^{\rho_{[i]}(\bphi)})\MA(\phi_{[i]})+\int_X \dot{\phi}_{[i]} \Delta_{[i]} \dot{\phi}_{[i]} \MA(\phi_{[i]}) \bigg] \\
&-& \lambda \int_X \bigg( \sum_i \dot{\phi}_{[i]} \bigg)^2 \mu(\bphi)+\lambda \bigg( \sum_i \int_X \dot{\phi}_{[i]} \mu(\bphi) \bigg)^2 \\
&=& -\sum_i \int_X (\ddot{\phi}_{[i]}-|\partial \dot{\phi}_{[i]}|^2)(1-e^{\rho_{[i]}(\bphi)})\MA(\phi_{[i]})\\
&-& \sum_i \int_X \dot{\phi}_{[i]} \big(\Delta_{[i]} \dot{\phi}_{[i]}+(\p \r_{[i]}, \p \dot{\phi}_{[i]})_{[i]} \big) \mu(\bphi)\\
&-& \lambda \int_X \bigg( \sum_i \dot{\phi}_{[i]} \bigg)^2 \mu(\bphi)+\lambda \bigg( \sum_i \int_X \dot{\phi}_{[i]} \mu(\bphi) \bigg)^2 \\
&=& -\sum_i \int_X (\ddot{\phi}_{[i]}-|\partial \dot{\phi}_{[i]}|^2)(1-e^{\rho_{[i]}(\bphi)})\MA(\phi_{[i]})-(\cP_{\bphi} \dot{\bphi}, \dot{\bphi})_{\bmu(\bphi)},
\end{eqnarray*}
where the first term is zero if $\bphi$ is a geodesic (\cf \cite[Proposition 4.25]{Sze14}) (but we will not use this formula in this paper). The following proposition is essentially due to \cite[Lemma 3, Lemma 4]{Hul17} and \cite[Theorem 1.3]{Pin18}, but we will give the proof for the sake of completeness:
\begin{prop} \label{wlaplacian}
The following properties hold:
\begin{enumerate}
\item The operator $\cP_{\bphi}$ is elliptic.
\item The operator $\cP_{\bphi}$ is self-adjoint with respect to the Hermitian product \eqref{lproduct}. In particular, all eigenvalues of $\cP_{\bphi}$ are real.
\item The operator $\cP_{\bphi}$ is non-positive and
\[
\Ker \cP_{\bphi}=\{\bv=(v_{[i]}) \in (C^\infty(X;\C))^N|V_{[1]}=\ldots=V_{[N]}, \text{$V_{[i]}$'s are holomorphic} \},
\]
where for each $\bv=(v_{[i]}) \in (C^\infty(X;\C))^N$, $V_{[i]}$ denotes a vector field of type $(1,0)$ given by
\[
i_{V_{[i]}} \omega_{[i]}=\sqrt{-1} \bp v_{[i]}.
\]
In particular, if $\lambda=-1$ or $G$ is trivial, then the kernel consists of $N$ constant functions, \ie $\Ker \cP_{\bphi} \simeq \C^N$.
\end{enumerate}
\end{prop}
\begin{proof}
The property (1) follows from the ellipticity of each $\Delta_{[i]}$. Also $\cP_{\bphi}$ is self-adjoint since each $\Delta_{[i]}+(\p \r_{[i]}, \p \cdot)_{[i]} \colon C^\infty(X;\C) \to C^\infty(X;\C)$ is self-adjoint with respect to the $L^2$-inner product defined by the measure $e^{\r_{[i]}}\MA(\phi_{[i]})=\mu(\bphi)$, which gives the proof of (2). Now let us prove (3) in the $\lambda=\pm 1$ cases separately.

\vspace{2mm}

\textbf{Case 1 ($\lambda=-1$):}
Assume that $\cP_{\bphi}(\bv)=\g \bv$ for some $\g \in \R$. By using \eqref{eqmeas} and integrating by parts, we get
\begin{eqnarray*}
\sum_j \int_X |\p \overline{v_{[j]}}|_{[j]}^2 \mu (\bphi)&=& -\sum_j \int_X \overline{\big( \Delta_{[j]} v_{[j]}+(\p \r_{[j]}, \p \overline{v_{[j]}})_{[j]} \big)} \cdot v_{[j]} \mu (\bphi)\\
&=& -\int_X \bigg| \sum_j v_{[j]} \bigg|^2 \mu (\bphi)+\bigg| \int_X \sum_j v_{[j]} \mu(\bphi) \bigg|^2-\g \int_X \sum_j |v_{[j]}|^2 \mu(\bphi)\\
&\leq& -\g \int_X \sum_j |v_{[j]}|^2 \mu(\bphi).
\end{eqnarray*}
Thus $\g \leq 0$ and $\g=0$ if and only if $v_{[i]}$'s are constants.

\vspace{2mm}

\textbf{Case 2 ($\lambda=1$):}
For $s \in [0,1]$, we define
\[
\cP_{s, \bphi}(\bv):=
\begin{pmatrix}
\Delta_{[1]} v_{[1]}+(\p \r_{[1]}, \p \overline{v_{[1]}})_{[1]}+s \sum_j v_{[j]}-s \sum_j \int_X v_{[j]}\mu(\bphi)\\
\vdots \\
\Delta_{[N]} v_{[N]}+(\p \r_{[N]}, \p \overline{v_{[N]}})_{[N]}+s \sum_j v_{[j]}-s \sum_j \int_X v_{[j]}\mu(\bphi)
\end{pmatrix}.
\]
So $\{\cP_{s, \bphi} \}_{s \in [0,1]}$ is a continuous path of elliptic operators from the non-positive operator $\cP_{0, \bphi}$ to $\cP_{1, \bphi}=\cP_{\bphi}$. Now we consider the kernel of $\cP_{s, \bphi}$: we assume that an element $\bv \in (C^\infty(X; \C))^N$ satisfies $\cP_{s, \bphi}(\bv)={\bf 0}$. In particular, we have
\begin{equation} \label{equivlap}
\Delta_{[1]} v_{[1]}+(\p \r_{[1]}, \p \overline{v_{[1]}})=\ldots=\Delta_{[N]} v_{[N]}+(\p \r_{[N]}, \p \overline{v_{[N]}})=-s \sum_j v_{[j]}+s \int_X \sum_j v_{[j]} \mu(\bphi).
\end{equation}
Applying $\nabla_{[i],\bar{m}}$ to the $i$-th equation of $\cP_{s, \bphi}(\bv)={\bf 0}$ and commuting derivatives, we get
\begin{eqnarray*}
0 &=&\nabla_{[i],\bar{m}} \bigg(\Delta_{[i]} v_{[i]}+(\p \r_{[i]}, \p \overline{v_{[i]}})_{[i]}+s \sum_j v_{[j]} \bigg) \\
&=& g_{[i]}^{s \bar{u}} \nabla_{[i],s} \nabla_{[i],\bar{u}} \nabla_{[i],\bar{m}} v_{[i]}-\Ric(\omega_{\phi_{[i]}})_{s \bar{m}} V_{[i]}^s+g_{[i]}^{s \bar{u}} \nabla_{[i],s} \nabla_{[i],\bar{m}} \r_{[i]} \cdot \nabla_{[i], \bar{u}} v_{[i]}\\
&+& g_{[i]}^{s \bar{u}} \nabla_{[i],s} \r_{[i]} \cdot \nabla_{[i],\bar{m}} \nabla_{[i], \bar{u}} v_{[i]}+s \sum_j \nabla_{[i], \bar{m}} v_{[j]}\\
&=& g_{[i]}^{s \bar{u}} \nabla_{[i],s} \nabla_{[i], \bar{u}} \nabla_{[i],\bar{m}} v_{[i]}+g_{[i]}^{s \bar{u}} \nabla_{[i],s} \r_{[i]} \cdot \nabla_{[i],\bar{m}} \nabla_{[i], \bar{u}} v_{[i]}-\sum_j g_{[j], s \bar{m}} V_{[i]}^s+s \sum_j \nabla_{[i], \bar{m}} v_{[j]}.
\end{eqnarray*}
Multipling $g_{[i]}^{h \bar{m}}\nabla_{[i], h} \overline{v_{[i]}} \mu(\bphi)$ and integrating by parts, we get
\begin{eqnarray*}
\int_X |\overline{\nabla}_{[i]} V_{[i]}|_{[i]}^2 \mu(\bphi) &=&
\int_X |\overline{\nabla}_{[i]} \overline{\nabla}_{[i]} v_{[i]}|_{[i]}^2 \mu(\bphi) \\
&=&\sum_j \bigg[ s \int_X (\p \overline{v_{[i]}}, \p \overline{v_{[j]}})_{[i]} \mu(\bphi)-\int_X |V_{[i]}|_{[j]}^2 \mu(\bphi) \bigg] \\
&=&\sum_j \bigg[ s \int_X |\p \overline{v_{[j]}}|_{[j]}^2 \mu(\bphi)-\int_X |V_{[i]}|_{[j]}^2 \mu(\bphi) \bigg],
\end{eqnarray*}
where we used \eqref{equivlap} in the last equality. Then we find that this integral is non-positive. Indeed, if we choose a normal coodinate $(z_1,\ldots, z_n)$ with respect to $\omega_{\phi_{[j]}}$ such that $\omega_{\phi_{[i]}}$ is diagonal with eigenvalues $\b_1,\ldots,\b_n$ at a point $\hat{x} \in X$, then we observe that
\begin{eqnarray*}
|(\p \overline{v_{[i]}}, \p \overline{v_{[j]}})_{[i]}| &=& \bigg| \sum_m \frac{1}{\b_m} \frac{\p \overline{v_{[i]}}}{\p z_m} \overline{\frac{\p \overline{v_{[j]}}}{\p z_m}} \bigg| \\
&\leq& \bigg( \sum_m \bigg| \frac{1}{\b_m} \frac{\bp v_{[i]}}{\bp z_{\bar m}} \bigg|^2 \bigg)^{1/2} \bigg( \sum_m \bigg| \frac{\p \overline{v_{[j]}}}{\p z_m} \bigg|^2 \bigg)^{1/2}\\
&=& |V_{[i]}|_{[j]} |\p \overline{v_{[j]}}|_{[j]}
\end{eqnarray*}
at $\hat{x}$. Thus combining with the Cauchy-Schwartz inequality we get
\begin{eqnarray*}
\int_X |\p \overline{v_{[j]}}|_{[j]}^2 \mu(\bphi)&=&\int_X (\p \overline{v_{[i]}}, \p \overline{v_{[j]}})_{[i]} \mu(\bphi)\\
&\leq& \int_X |V_{[i]}|_{[j]} |\p \overline{v_{[j]}}|_{[j]} \mu(\bphi)\\
&\leq& \bigg( \int_X |V_{[i]}|_{[j]}^2 \mu(\bphi) \bigg)^{1/2} \bigg( \int_X |\p \overline{v_{[j]}}|_{[j]}^2 \mu(\bphi) \bigg)^{1/2}.
\end{eqnarray*}
So we have
\begin{equation} \label{efk}
\int_X |\overline{\nabla}_{[i]} V_{[i]}|_{[i]}^2 \mu(\bphi) \leq (s-1) \sum_j \int_X |V_{[i]}|_{[j]}^2 \mu(\bphi).
\end{equation}
When $s \in [0,1)$, the inequality \eqref{efk} implies that $V_{[1]}=\ldots=V_{[N]}=0$, and hence $v_{[i]}$'s are constants. In particular, we find that $\Ker \cP_{s,\bphi}$ is stationary for $s \in [0,1)$, and isomorphic to the space of $N$ constant functions $\C^N$. Let $W$ be the orthogonal complement of $\C^N$ in $(C^\infty(X;\C))^N$. Then the restriction $\cP_{s,\bphi}|_W \colon W \longrightarrow W$ is invertible and $\cP_{0,\bphi}|_W$ is strictly negative. It follows that the first non-zero eigenvalue $\g_s \in \R$ of $\cP_{s,\bphi}$ ($s \in [0,1)$) does not pass through the origin (since if it occurs, the operator $\cP_{s,\bphi}|_W$ will be degenerate). Hence $\cP_{s,\bphi}$ is non-positive for $s \in [0,1]$. When the parameter $s$ reaches to one, we may have $\g_s \to 0$ and the inequality \eqref{efk} only implies that $V_{[i]}$'s are holomorphic for $s=1$. Then we set a function $\theta_{V_{[j]}}(\omega_{[i]}) \in C^\infty(X;\C)$ (uniquely determined modulo constant) by
\[
i_{V_{[j]}} \omega_{[i]}=\sqrt{-1} \bp \theta_{V_{[j]}}(\omega_{[i]}).
\]
In particular, we have $v_{[i]}= \theta_{V_{[i]}}(\omega_{[i]})$ (mod. const.) for $i=1,\ldots,N$. Since $V_{[1]}$ is holomorphic, applying $L_{V_{[1]}}$ to $e^{\rho_{[1]}(\bphi)} \MA(\phi_{[1]})=\mu(\bphi)$, we get
\[
\Delta_{[1]} v_{[1]}+(\p \r_{[1]},\p \overline{v_{[1]}})_{[1]}+\sum_j \theta_{V_{[1]}}(\omega_{[j]})=\text{(const)}.
\]
Subtracting from the first equation of $\cP_{\bphi}(\bv)={\bf 0}$, we get
\begin{equation} \label{elot}
\sum_{j}v_{[j]}=\sum_{j} \theta_{V_{[1]}}(\omega_{[j]}) \quad \text{(mod. const)}.
\end{equation}
Again, since $V_{[1]}$ is holomorphic, applying $L_{V_{[1]}}$ to $e^{\rho_{[2]}(\bphi)} \MA(\phi_{[2]})=\mu(\bphi)$ yields
\begin{equation} \label{seqlap}
\Delta_{[2]} \theta_{V_{[1]}}(\omega_{[2]})+(\p \r_{[2]},\p \overline{\theta_{V_{[1]}}(\omega_{[2]})})_{[2]}+\sum_j \theta_{V_{[1]}}(\omega_{[j]})=\text{(const)}.
\end{equation}
Set $f:=\theta_{V_{[1]}}(\omega_{[2]})-v_{[2]}$ so that the type $(1,0)$ vector field of $f$ with respect to $\omega_{[2]}$ is given by $V_{[1]}-V_{[2]}$. Substracting \eqref{seqlap} from the second equation of $\cP_{\bphi}(\bv)={\bf 0}$ and using \eqref{elot}, we get
\[
\Delta_{[2]}f+(\p \r_{[2]}, \p \overline{f})_{[2]}=\text{(const)}.
\]
Multiplying $\overline{f} \mu(\bphi)$ and integrating by parts, we find that $f=\text{(const)}$, and hence $V_{[1]}=V_{[2]}$. Similarly, one can see that all $V_{[i]}$'s coincide.
\end{proof}
\begin{rem}
In \cite[Theorem 3.3]{FZ18}, they proved a similar result under an extra assumption that all $V_{[i]}$'s coincide. However, the argument in Proposition \ref{wlaplacian} (based on \cite[Lemma 3, Lemma 4]{Hul17} and \cite[Theorem 1.3]{Pin18}) shows that this assumption is satisfied automatically. Since we do not know a priori whether $V_{[i]}$'s are holomorphic, it seems to be impossible to apply the argument in \cite[Theorem 3.3]{FZ18} directly to our case.
\end{rem}
\begin{rem}
From the second variational formula of $\cD$, the non-positivity of $\cP_{\bphi}$ can be seen as the convexity of Ding functional \cite{HN17} at the linearized level, which comes from a more general convexity property due to Berndtsson \cite{Bern15}.
\end{rem}
%==============Section 3==========================
\section{Geometric quantization} \label{Gquant}
%==============Subsection 3.1==========================
\subsection{Setup of spaces and functionals}
\subsubsection{For a single polarization} \label{sfspl}
First, let $(X,L)$ be a polarized K\"ahler manifold. We take a reference fiber metric $\hat{h}$ with positive curvature $\hat{\omega} \in 2\pi c_1(L)$. For any integer $k$, we define the space of {\it Bergman metrics} (at level $k$) by
\[
\cH^{(k)}:=\{H | \text{$H$ is a Hermitian form on $H^0(X,L^{\otimes k})$}\}.
\]
Set $\GL^{(k)}:=\GL(H^0(X,L^{\otimes k});\C)$. Then the group $\GL^{(k)}$ acts on $\cH^{(k)}$ and for any $H \in \cH^{(k)}$, we have an isomorphism
\[
\cH^{(k)} \simeq \GL^{(k)}/U_H^{(k)},
\]
where $U_H^{(k)}$ denotes the unitary subgroup with respect to $H$.
In particular, the tangent space at $H$ is isomorphic to $\sqrt{-1} \fu_H^{(k)}:=\sqrt{-1} \Lie(U_H^{(k)})$. The space $\cH^{(k)}$ has a natural structure of Riemannian symmetric space
\[
\langle A, B \rangle_H:=\frac{1}{k^2 D^{(k)}} \Tr (A \circ B), \quad A, B \in \sqrt{-1} \fu_H^{(k)},
\]
which defines a distance function denoted by $\dist^{(k)}$. Also we denote the geodesic ball of radius $C$ centered at $H$ by $B_H(C)$. If we further take an $H$-orthonormal basis (written as ONB for short) $(s_\a)$, then the elements $A$, $B$ are identified with matrices and the above product is just the trace of the matrix $AB$.
For any $s \in H^0(X,L^{\otimes k})$ and $\phi \in \cH$, we write
\[
|s|_{k \phi}^2:=|s|_{\hat{h}^{\otimes k}}^2 e^{-k\phi}
\]
for simplicity. Let $\nu$ be a probability measure on $X$. For $\phi \in \cH$, we define an element $\Hilb_\nu^{(k)}(\phi) \in \cH^{(k)}$ as the $L^2$-inner product with respect to $\phi$ and $\nu$:
\[
\|s\|_{\Hilb_\nu^{(k)}(\phi)}^2:=\int_X |s|_{k\phi}^2 \nu, \quad s \in H^0(X,L^{\otimes k}).
\]
Conversely, for any $H \in \cH^{(k)}$, we set
\[
\FS^{(k)}(H):=\frac{1}{k} \log \bigg( \frac{1}{D^{(k)}} \sum_{\a=1}^{D^{(k)}} |s_\a|_{k0} ^2 \bigg),
\]
where $(s_\a)$ is any $H$-ONB of $H^0(X,L^{\otimes k})$ and
\[
D^{(k)}:=\dim H^0(X, L^{\otimes k}).
\]
One can easily see that the definition of $\FS^{(k)}(H)$ does not depend on the choice of $(s_\a)$. For any $H \in \cH^{(k)}$, let $M(H) \colon \P^\ast(H^0(X, L^{\otimes k})) \to (\sqrt{-1}\fu_H^{(k)})^\ast \simeq \sqrt{-1} \fu_H^{(k)}$ be ($\sqrt{-1}$-times of) the moment map for the standard $U_H^{(k)}$-action on $\P^\ast(H^0(X, L^{\otimes k}))$ with respect to the Fubini-Study metric $\omega_{\aFS(H)} \in 2 \pi c_1(\P^\ast(H^0(X,L^{\otimes k})), \cO(1))$ determined by $H$ (where we distinguish $\omega_{\aFS(H)}$ from $\omega_{\FS^{(k)}(H)}$, the restriction of $\frac{1}{k} \omega_{\aFS(H)}$ to $X$). So if we take an $H$-ONB $(s_\a)$, $M(H)$ has a matrix representation
\[
M(H):= \frac{(s_\a, s_\b)}{\sum_{\g=1}^{D^{(k)}}|s_\g|^2}.
\]
For any $A \in \sqrt{-1} \fu_H^{(k)}$, Let $\xi_A$ be a vector field on $\P^\ast(H^0(X, L^{\otimes k}))$ generated by $A$. Then
\[
h_H(A):=\Tr(A \circ M(H))
\]
is a hamiltonian for $\xi_A$ with respect to $\omega_{\aFS(H)}$, \ie $i_{\xi_A} \omega_{\aFS(H)}=\sqrt{-1} \bp h_H(A)$.
Also we set
\[
\overline{M}_\nu (H):=\int_X M(H) \nu.
\]
Now fix a reference metric on each level
\[
\hat{H}^{(k)}:=\Hilb_{\MA(0)}^{(k)}(0),
\]
and define the {\it quantized Aubin-Mabuchi functional} as
\[
\AM^{(k)}(H)=-\frac{1}{kD^{(k)}} \log \det (H \cdot (\hat{H}^{(k)})^{-1}),
\]
where $H \cdot (\hat{H}^{(k)})^{-1}$ denotes the endmorphism of $H^0(X, L^{\otimes k})$, obtained by switching the type of tensors. Meanwhile, any geodesic in $\cH$ are called the {\it Bergman geodesic}, which is expressed as $H(t):=\exp(-tA) \cdot H$ for some $A \in \sqrt{-1} \fu_H^{(k)}$. Then we can take an $H(0)$-orthonormal and $H(t)$-orthogonal basis $(s_\a)$ so that the element $A$ can be expressed as a diagonal matrix $(A_{\a \a})$. A simple computation shows that
\[
\ddt \AM^{(k)}(H(t))=\frac{1}{kD^{(k)}} \sum_\a A_{\a \a}.
\]
In particular, the functional $\AM^{(k)}$ is affine along Bergman geodesics. With the inner product $\langle \cdot, \cdot \rangle_H$, the variation of $\AM^{(k)}$ is given by
\begin{equation} \label{dqam}
\d \AM^{(k)}|_H (A)=k \langle A, \Id \rangle_H.
\end{equation}
Finally, we define the quantized $J$-functional as
\[
J^{(k)}=-\AM^{(k)}+\hat{L} \circ \FS^{(k)}.
\]
\subsubsection{For coupled polarizations} \label{sfcc}
Let $X$ be a compact K\"ahler manifold with $\lambda c_1(X)>0$ with $\lambda=\pm 1$, and $(L_{[i]})_{i=1,\ldots,N}$ a rational decomposition of $K_X^{\otimes -\lambda}$, \ie each $L_{[i]}$ is a $\Q$-ample line bundle with an isomorphism
\begin{equation} \label{lbiso}
K_X^{\otimes -\lambda} \simeq \otimes_i L_{[i]}.
\end{equation}
We assume that each $L_{[i]}$ is a line bundle for simplicity. Let $(\hat{h}_{[i]})_{i=1,\ldots, N}$ be an $N$-tuple of Hermitian metrics on $(L_{[i]})$ with positive curvature $\hat{\omega}_{[i]} \in 2 \pi c_1(L_{[i]})$. Let $\bnu=(\nu_{[i]})_{i=1,\ldots,N}$ be an $N$-tuple of probability measures, where each $\nu_{[i]}$ is allowed to depend on $\bphi \in \prod_i \cH_{[i]}$. Write
\[
\Hilb_{[i],\bnu}^{(k)}(\bphi):=\Hilb_{[i],\nu_{[i]}(\bphi)}^{(k)}(\phi_{[i]})
\]
for simplicity and set
\[
\bHilb_{\bnu}^{(k)}(\bphi):=(\Hilb_{[1],\bnu}^{(k)}(\bphi), \ldots, \Hilb_{[N],\bnu}^{(k)}(\bphi)), \quad \bphi \in \prod_i \cH_{[i]}.
\]
On the other hand, the Fubini-Study map $\bFS^{(k)}$ is just defined by the product of $\FS_{[i]}^{(k)}$ ($i=1,\ldots, N$) as follows:
\[
\bFS^{(k)}(\bH):=(\FS_{[1]}^{(k)}(H_{[1]}),\ldots, \FS_{[N]}^{(k)}(H_{[N]})), \quad \bH=(H_{[i]}) \in \prod_i \cH_{[i]}^{(k)}.
\]
With these maps, we define
\[
S_{[i], \bnu}^{(k)}(\bphi):=\FS_{[i]}^{(k)}(\Hilb_{[i],\bnu}^{(k)}(\bphi)), \quad T_{[i], \bnu}^{(k)}(\bH):=\Hilb_{[i],\bnu}^{(k)}(\FS_{[i]}^{(k)}(\bH)),
\]
\[
\tilde{S}_{\bnu}^{(k)}(\bphi):=(S_{[1], \bnu}^{(k)}(\bphi), \ldots, S_{[1], \bnu}^{(k)}(\bphi)), \quad 
\tilde{T}_{\bnu}^{(k)}(\bH):=(T_{[1], \bnu}^{(k)}(\bH), \ldots, T_{[N], \bnu}^{(k)}(\bH)).
\]
We set
\[
\overline{M}_{[i], \bnu}(\bH):=\overline{M}_{[i], \nu_{[i]} (\bFS^{(k)}(\bH))}(H_{[i]}), \quad \bH \in \prod_i \cH_{[i]}^{(k)}.
\]
We remark that if the $N$-tuple of measure $\bnu$ does not depend on $\bphi$, then everything defined above completely decouples. We write the $N$ copy of the canonical measure as 
\[
\bmu=(\mu, \ldots, \mu).
\]
Also, for $\bphi=(\phi_{[i]}) \in \prod_i \cH_{[i]}$, we set
\[
\bMA(\bphi):=(\MA(\phi_{[1]}), \ldots, \MA(\phi_{[N]})).
\]
\begin{dfn}
We say that an $N$-tuple of Hermitian forms $\bH \in \prod_i \cH_{[i]}^{(k)}$ is balanced (at level $k$) if it satisfies
\begin{equation} \label{bmetric}
\tilde{T}_{\bmu}^{(k)}(\bH)=\bH.
\end{equation}
Here $\bmu=(\mu,\ldots,\mu)$ and the measure $\mu$ is defined by \eqref{canmeasure}. The balanced condition \eqref{bmetric} is equivalent to
\[
T_{[i],\bmu}^{(k)}(H_{[i]})=\Hilb_{[i],\bmu}^{(k)}(\FS_{[i]}^{(k)}(H_{[i]}))=H_{[i]}
\]
for all $i$, where $\bmu$ depends on all the $\FS_{[i]}^{(k)}(H_{[i]})$'s at once.
\end{dfn}
We also say that an $N$-tuple $\bphi \in \prod_i \cH_{[i]}$ is balanced at level $k$ if
\[
\tilde{S}_{\bmu}^{(k)} (\bphi)=\bphi.
\]
Then we have an isomorphism
\[
\bigg\{\bphi \in \prod_i \cH_{[i]} \bigg| \text{$\bphi$ is balanced at level $k$} \bigg\} \simeq \big\{\bH \in \prod_i \cH_{[i]}^{(k)} \bigg| \text{$\bH$ is balanced} \bigg\}
\]
given by the bijections $\bHilb_{\bmu}^{(k)}$ and $\bFS^{(k)}$.
\begin{rem}
Independent of one's interest, the iteration for $\tilde{T}_{\bmu}^{(k)}$ defines a dynamical system $\{ \big(\tilde{T}_{\bmu}^{(k)} \big)^s \}_{s=1,2,\ldots}$ on $\cH_{[i]}^{(k)}$ and balanced metrics are characterized as the unique fixed point modulo automorphisms. Also it seems to be a discretization of the balancing flow \eqref{balflow}.
\end{rem}
For $\bH \in \prod_i \cH_{[i]}^{(k)}$, we set
\[
\cJ^{(k)}(\bH):=\sum_i J_{[i]}^{(k)}(H_{[i]}),
\]
and define the {\it quantized Ding functional} by
\[
\cD^{(k)}(\bH):=-\sum_i \AM_{[i]}^{(k)}(H_{[i]})+\cL(\bFS^{(k)}(\bH)).
\]

Finally, we consider the case when $G$ is not trivial. Each $G$-linearization on $L_{[i]}$ induces a linear transformation on $H^0(X, L_{[i]}^{\otimes k})$ for sufficiently large $k$. We set
\[
\cJ_G^{(k)}(\bH):=\inf_{f \in G}\cJ^{(k)}(f^\ast \bH).
\]
\begin{prop} \label{qdf}
The quantized Ding functional $\cD^{(k)}$ satisfies the following properties:
\begin{enumerate}
\item The variational formula for $\cD^{(k)}$ is given by
\[
\d \cD^{(k)}|_{\bH}(\bA)=k \langle \bA, (D_{[i]}^{(k)} \overline{M}_{[i], \bmu}(\bH)-\Id) \rangle_{\bH}.
\]
In particular, an $N$-tuple $\bH \in \prod_i \cH_{[i]}^{(k)}$ is balanced if and only if it is a critical point of $\cD^{(k)}$.
\item $\cD^{(k)}$ is convex along Bergman geodesics, and to be affine if and only if it arises as a flow generated by a holomorphic vector field.
\item If $\bH, \bH^\dag \in \prod_i \cH_{[i]}^{(k)}$ are balanced, then there exists some $f \in G$ such that $\bH=f^\ast \bH^\dag$.
\end{enumerate}
\end{prop}
\begin{proof}
(1) Let $\bH(t) \in \prod_i \cH_{[i]}^{(k)}$ be a Bergman geodesic generated by $\bA=(A_{[i]}) \in \sqrt{-1} \big( \oplus_i \fu_{H_{[i]}(0)}^{(k)} \big)$. We take an $H_{[i]}(0)$-orthonormal, and $H_{[i]}(t)$-orthogonal basis $(s_{[i],\a})_{\a=1\ldots,D_{[i]}^{(k)}}$ so that $A_{[i]}$ is expressed as the diagonal matrix $(A_{[i],\a\a})$. Then $(\exp(A_{[i],\a\a}t/2) \cdot s_{[i],\a})_{\a=1\ldots,D_{[i]}^{(k)}}$ defines an $H_{[i]}(t)$-ONB. In this setup, we can compute
\[
\ddt \FS_{[i]}^{(k)}(H_{[i]}(t))|_{t=0}=\frac{1}{k} \frac{\sum_{\a} A_{[i],\a\a} |s_{[i],\a}|^2}{\sum_\b |s_{[i],\b}|^2},
\]
and hence
\begin{eqnarray*}
\ddt \cL(\bFS^{(k)} (\bH(t)))|_{t=0} &=& \sum_i \frac{1}{k} \sum_\a A_{[i],\a\a} \int_X \frac{|s_{[i],\a}|^2}{\sum_\b |s_{[i],\b}|^2} \mu(\bFS^{(k)}(\bH(0))\\
&=& k \langle \bA, (D_{[i]}^{(k)} \overline{M}_{[i], \bmu}(\bH(0))) \rangle_{\bH(0)}.
\end{eqnarray*}
Thus combining with \eqref{dqam}, we get we get the variational formula for $\cD^{(k)}$. In particular, the Hermitian form $\bH$ is a critical point of $\cD^{(k)}$ if and only if
\[
\overline{M}_{[i], \bmu}(\bH)=\frac{\Id}{D_{[i]}^{(k)}}, \quad i=1,\ldots, N.
\]
On the other hand, for any $H_{[i]}$-orthonormal basis $(s_{[i],\a})_{\a=1,\ldots, D_{[i]}^{(k)}}$, the matrix representation of $\overline{M}_{[i], \bmu}(\bH)$ is given by
\begin{eqnarray*}
\int_X \frac{(s_{[i],\a}, s_{[i], \b})}{\sum_\g |s_{[i],\g}|^2} \mu(\bFS^{(k)}(\bH)) &=& \frac{1}{D_{[i]}^{(k)}} \int_X (s_{[i],\a}, s_{[i], \b})_{k \FS_{[i]}^{(k)}(H_{[i]})} \mu(\bFS^{(k)}(\bH)) \\
&=& \frac{1}{D_{[i]}^{(k)}} (s_{[i],\a}, s_{[i], \b})_{T_{[i],\bmu}^{(k)}(\bH)}.
\end{eqnarray*}
So if we further take $(s_{[i],\a})_{\a=1,\ldots, D_{[i]}^{(k)}}$ to be $T_{[i],\bmu}^{(k)}(\bH)$-orthogonal, then the above matrix is diagonal. Hence the Hermitian form $\bH$ is a critical point of $\cD^{(k)}$ if and only if
\[
\|s_{[i],\a}\|_{T_{[i],\bmu}^{(k)}(\bH)}=1, \quad i=1,\ldots, N, \quad \a=1,\ldots, D_{[i]}^{(k)}.
\]
Since $(s_{[i],\a})_{\a=1,\ldots, D_{[i]}^{(k)}}$ is $H_{[i]}$-ONB, this implies that
\[
T_{[i], \bmu}^{(k)}(\bH)=H_{[i]},
\]
and hence $\tilde{T}_{\bmu}^{(k)}(\bH)=\bH$, so $\bH$ is balanced. Conversely, if $\bH$ is balanced, we have $\|s_{[i],\a}\|_{T_{[i], \bmu}^{(k)}(\bH)}=\|s_{[i],\a}\|_{H_{[i]}}=1$, and hence $\overline{M}_{[i], \bmu}(\bH)=\Id/D_{[i]}^{(k)}$ as desired.

(2) By a property of the Fubini-Study map (for instance, see \cite[Section 5.1]{Bern09}), each $\FS_{[i]}^{(k)}(H_{[i]}(t))$ defines a $\PSH(X,\hat{\omega}_{[i]})$-subgeodesic. In particular, $\sum_j \FS_{[j]}^{(k)}(H_{[i]}(t))$ is also a $\PSH(X, \sum_j \hat{\omega}_{[j]})$-subgeodesic, and we can apply Berndtsson's convexity theorem \cite[Theorem 1.1]{Bern15} to find that $\cD^{(k)}(\bH(t))$ is convex. Moreover, $\cD^{(k)}(\bH(t))$ is affine if and only if $\sum_j \FS_{[j]}^{(k)}(H_{[i]}(t))$ is a $\PSH(X, \sum_j \hat{\omega}_{[j]})$-geodesic generated by a holomorphic vector field $V$. Since each $\FS_{[i]}^{(k)}(H_{[i]}(t))$ is a $\PSH(X,\hat{\omega}_{[i]})$-subgeodesic, this occurs if and only if each $\FS_{[i]}^{(k)}(H_{[i]}(t))$ is a $\PSH(X,\hat{\omega}_{[i]})$-geodesic generated by the same holomorphic vector field $V$ (see the arguments in \cite[Lemma 4.4]{HN17}). Let $f_t$ denotes the flow generated by $V$. By using the fact that $f_t$ and $\FS_{[i]}^{(k)}$ commute because the Kodaira embedding is equivariant with respect to the $G$-action (\cf \cite[Lemma 2.25]{Has15}), we observe that
\[
\bomega_{\bFS^{(k)}(\bH(t))}=f_t^\ast \bomega_{\bFS^{(k)}(\bH(0))}=\bomega_{\bFS^{(k)}(f_t^\ast \bH(0))}.
\]
Since the Fubini-Study map $\FS_{[i]}^{(k)}$ is injective (\cf \cite{Has17}), we have $\bH(t)=f_t^\ast \bH(0)$.

(3) Since we already know that the metrics $\bH, \bH^\dag \in \prod_i \cH_{[i]}^{(k)}$ are critical points of $\cD^{(k)}$, if we take a Bergman geodesic $\bH(t)$ joining them, we know that $\cD^{(k)}(\bH(t))$ is constant. Hence by (2), we obtain $\bH=f^\ast \bH^\dag$ for some $f \in G$ as desired.
\end{proof}
Following \cite{DL19}, for any Bergman geodesic ray $\{\bH(t)\} \subset \prod_i \cH_{[i]}^{(k)}$, we define the {\it radial quantized Ding functional}
\[
\cD_{\rm rad}^{(k)}(\{\bH(t)\}):= \lim_{t \to \infty} \frac{ \cD^{(k)}(\bH(t))}{t}.
\]
\begin{prop} \label{cqd}
Assume there exists a balanced metric at level $k$. Then $\cD_{\rm rad}^{(k)}(\{\bH(t)\})$ is non-negative for any Bergman geodesic ray $\{\bH(t)\} \subset \prod_i \cH_{[i]}^{(k)}$ with equality holding if and only if $\bH(t)$ arises as a flow generated by a holomorphic vector field.
\end{prop}
\begin{proof}
If the initial metric $\bH(0)$ is balanced, by the convexity of $\cD^{(k)}$ along Bergman geodesics $\bH(t)$, we know that the ratio $\cD^{(k)}(\bH(t))/t$ is increasing in $t$, thus if we take $\bH(0)$ as a balanced metric, then we have
\[
\cD_{\rm rad}^{(k)}(\{\bH(t)\})= \lim_{t \to \infty} \frac{ \cD^{(k)}(\bH(t))}{t} \geq \lim_{t \searrow 0} \frac{ \cD^{(k)}(\bH(t))}{t}=0.
\]
Further assume that $\cD_{\rm rad}^{(k)}(\{\bH(t)\})=0$. Then we find that $\cD^{(k)}(\bH(t))=0$ for all $t \geq 0$, and by Proposition \ref{qdf} (2), we have $\bH(t)=f_t^\ast \bH(0)$ for some flow $f_t$ generated by a holomorphic vector field. We note that changing the reference metric only results in an overall additive constant to the functional $\cD^{(k)}$. Thus we have the desired statement for any initial metric $\bH(0) \in \prod_i \cH_{[i]}^{(k)}$\footnote{This property is not used in later arguments since we mainly deal with the case when the balanced metric exist. Howerver, as an obstruction, Proposition \ref{cqd} should be stated for Bergman geodesics emanating from any element in $\prod_i \cH_{[i]}^{(k)}$.}.
\end{proof}
Finally, we compute the Hessian of the quantized Ding functional. For $\bH=(H_{[i]}) \in \prod_i \cH_{[i]}^{(k)}$ and $\bA=(A_{[i]}) \in \sqrt{-1} \big( \oplus_i \fu_{H_{[i]}}^{(k)} \big)$, we consider the pointwise decomposition of $\xi_{A_{[i]}}$:
\begin{equation} \label{pwdcomp}
\xi_{A_{[i]}}=\xi_{A_{[i]}}^\top+\xi_{A_{[i]}}^\bot,
\end{equation}
where $\xi_{A_{[i]}}^\top$ denotes the component of $\xi_{A_{[i]}}|_X$ which is tangent to $X$ and $\xi_{A_{[i]}}^\bot$ the component with is perpendicular to $X$ with respect to the Fubini-Study metric $\omega_{\aFS(H_{[i]})} \in 2 \pi c_1(\P^\ast(H^0(X,L_{[i]}^{\otimes k})), \cO(1))$. Then we have the following:
\begin{prop} \label{hessfqd}
We have
\[
\nabla^2 \cD^{(k)}|_{\bH}(\bA,\bA)=k^{-1} \sum_i \int_X |\xi_{A_{[i]}}^\bot|_{\aFS(H_{[i]})}^2 \mu(\bFS^{(k)}(\bH))-k^{-2}(\cP_{\bFS^{(k)}(\bH)} (h_{A_{[i]}}), (h_{A_{[i]}}))_{\bmu(\bFS^{(k)}(\bH))},
\]
where $h_{A_{[i]}}$ denotes the Hamiltonian with respect to $\xi_{A_{[i]}}$ defined in Section \ref{sfspl}.
\end{prop}
\begin{proof}
The following argument is a generalization of \cite[Corollary 1.1, Lemma 3.2]{Tak18}. Let $\bH(t) \in \prod_i \cH_{[i]}^{(k)}$ be a Bergman geodesic emanating from $\bH$ and generated by $\bA$. Then
\[
\nabla^2 \cD|_{\bH}(\bA,\bA)=\frac{d^2}{dt^2} \cD(\bH(t))|_{t=0}=\frac{d^2}{dt^2} \cL(\bFS^{(k)}(\bH(t)))|_{t=0},
\]
where we used the fact that $\bH(t)$ is a Bergman geodesic in the first equality and $\AM_{[i]}^{(k)}(H_{[i]}(t))$ is affine in $t$ in the second equality. From the proof of Proposition \ref{qdf}, we find that
\begin{equation} \label{firstd}
\frac{d}{dt} \cL(\bFS^{(k)}(\bH(t)))=k\langle \dot{\bH}(t), (D_{[i]}^{(k)} \overline{M}_{[i],\mu}(\bH(t))) \rangle_{\bH(t)}.
\end{equation}
Now we take an $H_{[i]}(0)$-orthonormal, $H_{[i]}(t)$-orthogonal basis $(s_{[i],\a})_{\a=1\ldots,D_{[i]}^{(k)}}$ as in the proof of Proposition \ref{qdf} (1), and consider the matrix representation for the RHS of \eqref{firstd} with respect to this basis. So $\dot{\bH}(t)$ and $\overline{M}_{[i],\bmu}(\bH(t))$ mean $N$-tuple of matrices and the bracket is just given by taking the trace of them. Also we note that $\dot{\bH}(t)=\bA$ is a constant diagonal matrix since $\bH(t)$ is a Bergman geodesic. Since $\overline{M}_{[i],\bmu}(\bH(t))=\int_X M(H_{[i]}(t)) \mu(\bFS^{(k)}(\bH(t)))$, differentiating in $t$, we get
\begin{eqnarray*}
\frac{d^2}{dt^2} \cL(\bFS^{(k)}(\bH(t)))|_{t=0} &=& k^{-1} \sum_i \bigg[ \int_X \Tr \big(A_{[i]} \ddt M(H_{[i]}(t))|_{t=0} \big) \mu(\bFS^{(k)}(\bH))\\
&+& \int_X h_{A_{[i]}} \ddt \mu(\bFS^{(k)}(\bH(t)))|_{t=0} \bigg],
\end{eqnarray*}
where we used the formula $h_{A_{[i]}}=\Tr(A_{[i]} M(H_{[i]}))$. Then the first integrand is
\[
\Tr (A_{[i]} \ddt M(H_{[i]}(t))|_{t=0} \big)=\Tr (A_{[i]} d(M(H_{[i]})(\xi_{A_{[i]}})))=|\xi_{A_{[i]}}|_{\aFS(H_{[i]})}^2,
\]
where we used the fact that $M(H_{[i]})$ is the moment map for $\omega_{\aFS(H_{[i]})}$ in the last equality (we also note that the exterior derivative $d$ in the above equation is defined on the ambient projective space). Since $\ddt \FS_{[i]}^{(k)}(H_{[i]}(t))|_{t=0}=\frac{1}{k} h_{A_{[i]}}$, by using Proposition \ref{vfmu}, we compute the second integrand as
\[
\ddt \mu(\bFS^{(k)}(\bH))|_{t=0}=\frac{\lambda}{k} \sum_i \bigg[ -h_{A_{[i]}}+\int_X h_{A_{[i]}} \mu(\bFS^{(k)}(\bH)) \bigg] \mu(\bFS^{(k)}(\bH)).
\]
Thus we have
\begin{eqnarray} \label{estqd}
\frac{d^2}{dt^2} \cL(\bFS^{(k)}(\bH(t)))|_{t=0}&=&k^{-1} \sum_i \int_X |\xi_{A_{[i]}}|_{\aFS(H_{[i]})}^2 \mu(\bFS^{(k)}(\bH)) \nonumber \\
&-& \lambda k^{-2} \int_X \bigg( \sum_i h_{A_{[i]}} \bigg)^2 \mu(\bFS^{(k)}(\bH)) \nonumber \\
&+& \lambda k^{-2} \bigg( \int_X \sum_i h_{A_{[i]}} \mu(\bFS^{(k)}(\bH)) \bigg)^2.
\end{eqnarray}
To deal with the first term of \eqref{estqd}, we use the trivial decomposition
\[
|\xi_{A_{[i]}}|_{\aFS(H_{[i]})}^2=|\xi_{A_{[i]}}^\top|_{\aFS(H_{[i]})}^2+|\xi_{A_{[i]}}^\bot|_{\aFS(H_{[i]})}^2,
\]
where $| \cdot |_{\aFS(H_{[i]})}^2$ denotes the norm measured by the Fubini-Study metric $\omega_{\aFS(H_{[i]})} \in 2\pi c_1(\P^\ast(H^0(X,L_{[i]}^{\otimes k})), \cO(1))$. Also, over $X$ we have
\[
k^{-1} |\xi_{A_{[i]}}^\top|_{\aFS(H_{[i]})}^2=|\xi_{A_{[i]}}^\top|_{\FS_{[i]}^{(k)}(H_{[i]})}^2=k^{-2} |\p h_{A_{[i]}}|_{\FS_{[i]}^{(k)}(H_{[i]})}^2,
\]
where we note that $\omega_{\FS_{[i]}^{(k)}(H_{[i]})}=k^{-1} \omega_{\aFS(H_{[i]})}$ on $X$. Thus, integrating by parts, the first term of \eqref{estqd} becomes
\begin{eqnarray*}
&& k^{-1}\sum_i \int_X |\xi_{A_{[i]}}^\bot|_{\aFS(H_{[i]})}^2 \mu(\bFS^{(k)}(\bH))\\
&& -k^{-2} \sum_i \int_X (\Delta_{\FS_{[i]}^{(k)}(H_{[i]})} h_{A_{[i]}}+(\p \r_{[i]}(\FS_{[i]}^{(k)}(H_{[i]})), \p h_{A_{[i]}})_{\FS_{[i]}^{(k)}(H_{[i]})}) h_{A_{[i]}} \mu(\bFS^{(k)}(\bH)).
\end{eqnarray*}
Putting this into \eqref{estqd}, the sum of all terms except $k^{-1} \sum_i \int_X |\xi_{A_{[i]}}^\bot|_{\aFS(H_{[i]})}^2 \mu(\bFS^{(k)}(\bH))$ is just equal to $-k^{-2}(\cP_{\bFS^{(k)}(\bH)} (h_{A_{[i]}}), (h_{A_{[i]}}))_{\bmu(\bFS^{(k)}(\bH))}$. This completes the proof.
\end{proof}
\begin{rem}
The following argument is pointed out by the referees and argued in \cite[Proposition 2.7]{ST16} when $N=1$: by \eqref{dpqdf}, we observe that
\[
\nabla^2 \cD^{(k)}|_{\bH}(\bA,\bA)=\frac{d^2}{dt^2} \cD^{(k)}(\bH(t))|_{t=0}=\sum_i \frac{d^2}{dt^2} Z_{[i]}^{(k)}(H_{[i]}(t))_{t=0}+\frac{d^2}{dt^2} \cD(\bFS^{(k)}(\bH(t))|_{t=0}.
\]
Since the Bergman geodesic ray $\bH(t)$ defines a subgeodesic ray $\bFS^{(k)}(\bH(t))$ and $\cD$ is convex along subgeodesics (\cf \cite[Section 5]{HN17}), the derivative $\frac{d^2}{dt^2} \cD(\bFS^{(k)}(\bH(t))|_{t=0}$ is non-negative. Moreover, it was proved in \cite[Lemma 17]{Fin10} that
\[
\frac{d^2}{dt^2} Z_{[i]}^{(k)}(H_{[i]}(t))|_{t=0}=k^{-1} \int_X |\xi_{A_{[i]}}^\bot|_{\aFS(H_{[i]})}^2 \MA(\FS_{[i]}^{(k)}(H_{[i]})).
\]
Hence we get
\begin{equation} \label{ehdima}
\nabla^2 \cD^{(k)}|_{\bH}(\bA,\bA) \geq k^{-1} \sum_i \int_X |\xi_{A_{[i]}}^\bot|_{\aFS(H_{[i]})}^2 \MA(\FS_{[i]}^{(k)}(H_{[i]})),
\end{equation}
that is enough to show Theorem \ref{crconv}. Indeed, we can use \eqref{ehdima} instead of Proposition \ref{hessfqd} to estimate $\frac{d^2}{dt^2} \cR^{(k)}(\bH(t))$ along the balancing flow (\cf Lemma \ref{cobf}).
\end{rem}
%==============Subsection 3.2==========================
\subsection{Algebraic obstruction to the existence of balanced metrics}
In this subsection, we always assume that $\lambda=1$. We first consider a single polarization $L \to X$.
\begin{dfn}
A test configuration (of exponent $k$) for $(X, L)$ consists of a polarized scheme $\cL \to \cX$ (where $\cL$ is allowed to be relatively ample) with $\cX$ normal, and satisfying the following data:
\begin{enumerate}
\item A $\C^\ast$-action on $\cX$ lifting to actions on $\cL$.
\item A flat $\C^\ast$-equivariant map $\pi \colon \cX \to \C$.
\item An isomorphism $(\cX, \cL)|_1 \simeq (X, L^{\otimes k})$ on the $1$-fiber.
\end{enumerate}
\end{dfn}
Denote $D^{(km)}$ the dimension of $H^0(\cX_0, m\cL_0)$ and $w^{(km)}$ the total weight of the $\C^\ast$-action on $H^0(\cX_0, m\cL_0)$. Then for a large integer $m$, we have expansions:
\begin{align*}
D^{(km)}=a_0(km)^n+a_1(km)^{n-1}+\cdots+a_n,\\
w^{(km)}=e_0(km)^{n+1}+e_1(km)^n+\cdots+e_{n+1}.
\end{align*}
Then the {\it Chow weight} of $(\cX,\cL)$ is defined by
\[
\Chow^{(k)}(\cX,\cL):=\frac{e_0}{a_0}-\frac{w^{(k)}}{k D^{(k)}}.
\]
If $(\cX,\cL)$ is a product configuration generated by a holomorphic vector field $V$, using the equivariant Riemann-Roch theorem, we can write down these coefficients explicitly as integral invariants. In particular, we observe that
\begin{equation} \label{hof}
\Chow^{(km)}(\cX,m\cL)=-\frac{1}{km D^{(km)}} \sum_{j=1}^n \frac{(km)^{n+1-j}}{(n+1-j)!} \cF_j(V),
\end{equation}
where $\cF_1,\ldots, \cF_n$ denotes the {\it higher order Futaki invariants} (\cf \cite{Fut04}), and the above formula was studied in \cite[Proposition 2.2]{VD12}.

Next we consider the coupled settings.
\begin{dfn}
A test configuration (of exponent $k$) for $(X, (L_{[i]}))$ consists of a normal scheme $\cX$ polarized by an $N$-tuple $(\cL_{[1]}, \ldots, \cL_{[N]})$ with the following data:
\begin{enumerate}
\item A $\C^\ast$-action on $\cX$ lifting to actions on $(\cL_{[1]}, \ldots, \cL_{[N]})$.
\item A flat $\C^\ast$-equivariant map $\pi \colon \cX \to \C$.
\item Isomorphisms $(\cX, \cL_{[1]}, \cdots, \cL_{[N]}, \otimes_i \cL_{[i]})|_1 \simeq (X, L_{[1]}^{\otimes k}, \ldots, L_{[N]}^{\otimes k}, K_X^{\otimes -k})$ on the $1$-fiber.
\end{enumerate}
\end{dfn}
In particular, each $(\cX,\cL_{[i]})$ defines a test configuration of exponent $k$ for $(X, L_{[i]})$ and associated Bergman geodesic ray $\{H_{[i]}(t)\} \in \cH_{[i]}^{(k)}$ for each $i=1,\ldots, N$. Then we want to compute the radial quantized Ding functional along the ray $\bH(t)=(H_{[i]}(t))$. Set
\[
Z_{[i]}^{(k)}:=\AM_{[i]} \circ \FS_{[i]}^{(k)}-\AM_{[i]}^{(k)}.
\]
Then we can decompose $\cD^{(k)}$ as
\begin{equation} \label{dpqdf}
\cD^{(k)}=\sum_i Z_{[i]}^{(k)}+\cD \circ \bFS^{(k)}.
\end{equation}
It is well known that the asymptotic slope of $ Z_{[i]}^{(k)}(H_{[i]}(t))$ is the Chow weight (\cf \cite[Proposition 3]{Don05}, or \cite[Theorem 2.9]{Mum77} and \cite[Theorem 1]{Pau04}):
\[
\lim_{t \to \infty} \frac{Z_{[i]}^{(k)}(H_{[i]}(t))}{t}=\Chow^{(k)}(\cX,\cL_{[i]}).
\]
On the other hand, the asymptotic slope of $\cD(\bFS^{(k)}(H_{[i]}(t)))$ is expressed by means of the Donaldson-Futaki invariant as follows: using the $\C^\ast$-equivariant compactification $(\bar{\cX}, (\bar{\cL}_{[i]}))$ over $\P^1$ (\ie we glue $(\bar{\cX}, (\bar{\cL}_{[i]}))$ with a trivial family around $\infty \in \P^1$), the {\it Donaldson-Futaki invariant} $\DF(\cX, (\cL_{[i]}))$ for the test configuration $(\cX, (\cL_{[i]}))$ (of exponent $k$) \cite[Section 5]{HN17} is defined by
\[
\DF(\cX, (\cL_{[i]})):=-\sum_i \frac{(\bar{\cL}_{[i]})^{n+1}}{(n+1)k^{n+1} L_{[i]}^n}+\frac{(K_{\bar{\cX}/\P^1}+\frac{1}{k} \sum_i \bar{\cL}_{[i]})\cdot (\sum_i \bar{\cL}_{[i]})^n}{k^n (-K_X)^n}.
\]
Let $\Delta \subset \C$ be a closed unit disk centered at the origin. Then each $\FS_{[i]}^{(k)}(H_{[i]}(t))$ defines an $\S^1$-invariant locally bounded metric on $\cL_{[i]}|_\Delta \to \Delta$ with positive curvature current (for instance, see \cite[Lemma 2.4]{ST16}). The following argument is essentially due to \cite[Section 5]{HN17}, but we will explain here for the completeness: we further decompose $\cD(\bFS^{(k)}(\bH(t)))$ as
\begin{equation} \label{dcomfD}
\cD(\bFS^{(k)}(\bH(t)))=-\sum_i \AM_{[i]}(\FS_{[i]}^{(k)}(H_{[i]}(t)))+\AM \big( \sum_i \FS_{[i]}^{(k)}(H_{[i]}(t)) \big)+D \big(\sum_i \FS_{[i]}^{(k)}(H_{[i]}(t)) \big),
\end{equation}
where $\AM$ (resp. $D$) denotes the Aubin-Mabuchi functional (resp. Ding functional) defined for the anti-canonical polarization. Therefore we can apply the slope formula of $D$ (\cf \cite[Theorem 3.11]{Berm16}) to the test configuration $(\cX, \otimes_i \cL_{[i]})$ for $(X,K_X^{\otimes -1})$ with the induced fiber metric on $\otimes_i \cL_{[i]}$, and compute the asymptotic slope of the third term in \eqref{dcomfD} as
\[
\lim_{t \to \infty} \frac{D \big( \sum_i \FS_{[i]}^{(k)}(H_{[i]}(t)) \big)}{t}=-\frac{(\sum_i \bar{\cL}_{[i]})^{n+1}}{(n+1)k^{n+1} (-K_X)^n}+\frac{(K_{\bar{\cX}/\P^1}+\frac{1}{k} \sum_i \bar{\cL}_{[i]})\cdot (\sum_i \bar{\cL}_{[i]})^n}{k^n (-K_X)^n}-q,
\]
where the number $q$ is a non-negative and rational determined by the central fiber $\otimes_i \cL_{[i]}|_0$. The quantity $q$ vanishes if and only if $\cX$ is $\Q$-Gorenstein with $\otimes_i \cL_{[i]}$ isomorphic to the relative canonical bundle $K_{\cX/\C}^{\otimes -k}$ and $\cX_0$ reduced, and its normalization has at worst klt singularities\footnote{We need normality of the total space $\cX$ to apply \cite[Theorem 3.11]{Berm16}. Also we note that the normality of $\cX$ may not imply the normality of $\cX_0$.}. On the other hand, for the first and second term of \eqref{dcomfD}, one can apply the slope formula for Aubin-Mabuchi energy (\cf \cite[Theorem 4.2]{BHJ19}) to get
\[
\lim_{t \to \infty} \frac{\AM_{[i]}(\FS_{[i]}^{(k)}(H_{[i]}(t)))}{t}=\frac{(\bar{\cL}_{[i]})^{n+1}}{(n+1)k^{n+1} L_{[i]}^n},
\]
\[
\lim_{t \to \infty} \frac{\AM \big( \sum_i \FS_{[i]}^{(k)}(H_{[i]}(t)) \big)}{t}=\frac{(\sum_i \bar{\cL}_{[i]})^{n+1}}{(n+1)k^{n+1} (-K_X)^n}.
\]
\begin{dfn}
For a test configuration $(\cX,(\cL_{[i]}))$ (of exponent $k$) for $(X, (L_{[i]}))$, we define the quantized Donaldson-Futaki invariant
\[
\DF^{(k)}(\cX,(\cL_{[i]})):=\sum_i \Chow^{(k)}(\cX, \cL_{[i]})+\DF(\cX,(\cL_{[i]})).
\]
\end{dfn}
Then the above argument already shows that
\begin{equation} \label{sfmlgD}
\lim_{t \to \infty} \frac{\cD(\bFS^{(k)}(\bH(t)))}{t}=\DF^{(k)}(\cX,(\cL_{[i]}))-q.
\end{equation}
We note that the invariant $\DF(\cX, (\cL_{[i]}))$ is unchanged by replacing $(\cL_{[i]})$ with its power $(m\cL_{[i]})$, while we have
\[
\Chow^{(km)}(\cX, m \cL_{[i]})=\frac{e_{[i], 0}}{a_{[i], 0}}-\frac{w_{[i]}^{(km)}}{km D_{[i]}^{(km)}}=O(m^{-1})
\]
as $m \to \infty$. Hence we have
\begin{equation} \label{asinv}
\lim_{m \to \infty} \DF^{(km)}(\cX,(m \cL_{[i]}))=\DF(\cX,(\cL_{[i]})).
\end{equation}
\begin{dfn}
We say that:
\begin{enumerate}
\item $(X, (L_{[i]}))$ is semistable at level $k$ if $\DF^{(k)}(\cX,(\cL_{[i]})) \geq 0$ for all test configurations $(\cX,(\cL_{[i]}))$ of exponent $k$.
\item $(X, (L_{[i]}))$ is polystable at level $k$ if it is semistable at level $k$ and the equality $\DF^{(k)}(\cX,(\cL_{[i]}))=0$ holds for a test configuration $(\cX,(\cL_{[i]}))$ of exponent $k$ if and only if $(\cX,(\cL_{[i]}))$ is product.
\item $(X, (L_{[i]}))$ is asymptotically (semi/poly)stable if there exists $k_0 \in \N$ such that $(X, (L_{[i]}))$ is (semi/poly)stable at level $k$ for all $k \geq k_0$.
\end{enumerate}
\end{dfn}
Then from Proposition \ref{cqd}, \eqref{sfmlgD} and \eqref{asinv}, we have the following:
\begin{thm} \label{sims}
If $(X, (L_{[i]}))$ admits a balanced metric at level $k$, then it is polystable at level $k$. Moreover, the asymptotic semistability implies K-semistability, \ie $DF(\cX, (\cL_{[i]})) \geq 0$ holds for all test configurations $(\cX, (\cL_{[i]}))$.
\end{thm}
\begin{rem}
Theorem \ref{sims} combined with Theorem \ref{balp} implies that if $(X,(L_{[i]}))$ admits a CKE metric, then it is K-semistable. In fact, a stronger result proving K-polystability already appeared in \cite[Theorem D]{HN17}.
\end{rem}
Now we focus on the case when $(\cX, (\cL_{[i]}))$ is a product configuration generated by a holomorphic vector field $V$. In this case, the Donaldson-Futaki invariant coincides with the coupled Futaki invariant \cite[Section 3.2]{DP19}:
\begin{equation} \label{coupledFut}
\DF(\cX, (\cL_{[i]}))=\Fut_c(V):=-\sum_i \frac{1}{V_{[i]}} \int_X \theta_V(\omega_{[i]}) \omega_{[i]}^n+\sum_i \int_X \theta_V(\omega_{[i]}) \mu(\bomega),
\end{equation}
where $V_{[i]}^n:=(2\pi L_{[i]})^n$, $\bomega=(\omega_{[i]}) \in \prod_i (\cH_{[i]}/\R)$ denotes any ${\rm Im}(V)$-invariant metric and $\theta_V(\omega_{[i]})$ is a Hamiltonian, \ie $i_V \omega_{[i]}=\sqrt{-1} \bp \theta_V(\omega_{[i]})$. The variational formula \eqref{vfding} combined with \eqref{coupledFut} yields that
\begin{equation} \label{vfhdf}
\ddt \cD(\exp(t{\rm Re}V)^\ast \bomega)=\DF(\cX, (\cL_{[i]})).
\end{equation}
Also the equation \eqref{dpqdf} combined with \eqref{sfmlgD} yields that
\begin{equation} \label{vfhqdf}
\ddt \cD^{(k)}(\exp(t{\rm Re}V)^\ast \bomega)=\DF^{(k)}(\cX, (\cL_{[i]}))
\end{equation}
since the LHS does not depend on $t$ by Proposition \ref{qdf} (2), and $q=0$ (in the equation \eqref{sfmlgD}) for product configurations $(\cX, (\cL_{[i]}))$. For any sufficiently large integer $m$, we consider the asymptotic expansion of $\DF^{(k)}(\cX,(m\cL_{[i]}))$:
\begin{eqnarray*}
f(km)&:=& km D_{[1]}^{(km)} \cdots D_{[N]}^{(km)} \cdot  \DF^{(k)}(\cX,(m\cL_{[i]})) \\
&=& km D_{[1]}^{(km)} \cdots D_{[N]}^{(km)} \Fut_c(V)\\
&-& \sum_i D_{[1]}^{(km)} \cdots \elm{D_{[i]}^{(km)}} \cdots D_{[N]}^{(km)} \sum_{j=1}^n \frac{(km)^{n+1-j}}{(n+1-j)!} \cF_{[i], j} (V),
\end{eqnarray*}
where $\cF_{[i], j}(V)$ denotes the $j$-th order Futaki invariant with respect to a product configuration $(\cX, \cL_{[i]})$, and the notion $\elm{T}$ means eliminating the term $T$. So $f(km)$ is a polynomial of $km$ of degree $nN+1$.
\begin{dfn}
We define the higher order coupled Futaki invariants $\cF_{c, j}$ as the coefficients of $f$:
\[
f(km)=\cF_{c,1}(V)(km)^{nN+1}+\cdots+\cF_{c, nN+1}(V) km.
\]
\end{dfn}
If $(X, (L_{[i]}))$ is asymptotically semistable, then we have $\DF^{(k)}(\cX,(m\cL_{[i]}))=0$ for all positive integers $m$ and product configurations $(\cX,(\cL_{[i]}))$. Thus the invariants $\cF_{c, j}$ defines an obstruction to the asymptotic semistability:
\begin{thm}
If $(X, (L_{[i]}))$ is asymptotically semistable, then we have $\cF_{c, j} \equiv 0$ for all $j=1, \ldots, nN+1$.
\end{thm}
By definition, each $\cF_{c, j}$ can be written as a linear combination of $\Fut_c$ and $\{\cF_{[i], \ell}\}_{\ell=1,\ldots,n, \; i=1,\ldots, N}$ whose coefficients $\{a_{[i], s}\}_{s=0,\ldots,n, \; i=1,\ldots, N}$ comes from those of the Hilbert polynomials of $\{(X, L_{[i]})\}_{i=1,\ldots, N}$. For instance,
\[
\cF_{c,1}=\prod_i a_{[i], 0} \cdot \Fut_c,
\]
\[
\cF_{c,2}=\sum_i a_{[1],0} \cdots \elm{a_{[i], 0}} \cdots a_{[N], 0} \cdot a_{[i], 1} \cdot \Fut_c - \frac{1}{n!} \sum_i a_{[1], 0} \cdots \elm{a_{[i], 0}} \cdots a_{[N], 0} \cF_{[i], 1}.
\]
In particular, the invariant $\cF_{c,1}$ is a positive multiple of $\Fut_c$. Also we remark that $\DF^{(k)}(\cX, (m\cL_{[i]}))$ vanishes for all product configurations $(\cX, (\cL_{[i]}))$ and sufficiently large integer $m$ if and only if $\cF_{c, j} \equiv 0$ for all $j=1, \ldots, nN+1$ by the definition of $f$. Combining with \eqref{vfhdf} and \eqref{vfhqdf}, we can conclude the following:
\begin{prop} \label{Ginvariance}
The functional $\cD$ is invariant under the $G$-action if and only if $\Fut_c$ (or $\cF_{c,1}$) vanishes. Also the quantized functional $\cD^{(k)}$ are invariant under the $G$-action for sufficiently large $k$ if and only if $\cF_{c, j} \equiv 0$ for all $j=1, \ldots, nN+1$.
\end{prop}
%==============Subsection 4==========================
\section{Existence and weak convergence of balanced metrics} \label{ewcbm}
We prove Theorem \ref{balp}. The proof is essentially the same as the case $N=1$ (\cf \cite{BBGZ13, BN14, Tak15}) once the comparison formula between $\cJ$ and $\cJ^{(k)}$ is established (see \eqref{cfjf} in the proof).
\begin{proof}[Proof of Theorem \ref{balp}]
Let $\lambda=1$ (but the same proof works for the case $\lambda=-1$ although the $C^\infty$-convergence result is given in Theorem \ref{crconv}). We assume that the manifold $X$ admits a CKE metric with $\cF_{c, 1} \equiv \ldots \equiv \cF_{c, nN+1} \equiv 0$. In particular, by Proposition \ref{Ginvariance}, we know that the functionals $\cD$ and $\cD^{(k)}$ are invariant under the $G$-action. Then the coercivity of $\cD$ (\cf Theorem \ref{Coercivity}) shows that there exist some constants $\d, C>0$ such that for any $\bH \in \prod_i \cH_{[i]}^{(k)}$ and $f \in G$, we have
\[
\cD(\bFS^{(k)}(\bH)) \geq \d \cJ(\bFS^{(k)}(f^\ast \bH))-C.
\]
We can compute the LHS as
\begin{eqnarray*}
\cD(\bFS^{(k)}(\bH))
&=& \cD(f^\ast \bomega_{\bFS^{(k)}(\bH)})\\
&=&\cD(\bFS^{(k)}(f^\ast \bH))\\
&=& \cJ(\bFS^{(k)}(f^\ast \bH))+(\cL-\hat{\cL})(\bFS^{(k)}(f^\ast \bH)).
\end{eqnarray*}
Hence we have
\[
(1-\d) \cJ(\bFS^{(k)}(f^\ast \bH))+(\cL-\hat{\cL})(\bFS^{(k)}(f^\ast \bH)) \geq -C.
\]
We compare $\cJ$ with the exhaustion function $\cJ^{(k)}$ on a finite dimensional space by using the following lemma:
\begin{lem}[See \cite{BBGZ13}, Lemma 7.7]
There exists a sequence $\d_{[i]}^{(k)} \to 0$ such that
\[
J_{[i]} \circ \FS_{[i]}^{(k)} \leq (1+\d_{[i]}^{(k)}) J_{[i]}^{(k)}+\d_{[i]}^{(k)}.
\]
\end{lem}
Since $J_{[i]}$ is non-negative, the above lemma in particular implies that
\[
\inf_{\cH_{[i]}^{(k)}} J_{[i]}^{(k)} \geq -\frac{\d_{[i]}^{(k)}}{1+\d_{[i]}^{(k)}}=:-\e_{[i]}^{(k)}
\]
with $\e_{[i]}^{(k)} \to 0$. Hence we have
\begin{eqnarray*}
J_{[i]} \circ \FS_{[i]}^{(k)} &\leq& (1+\d_{[i]}^{(k)}) J_{[i]}^{(k)}+\d_{[i]}^{(k)}\\
&=& (1+\d_{[i]}^{(k)})(J_{[i]}^{(k)}+\e_{[i]}^{(k)})+\d_{[i]}^{(k)}-(1+\d_{[i]}^{(k)}) \e_{[i]}^{(k)}\\
&\leq& (1+\d^{(k)})(J_{[i]}^{(k)}+\e_{[i]}^{(k)})+\d_{[i]}^{(k)}-(1+\d_{[i]}^{(k)}) \e_{[i]}^{(k)}\\
&=&  (1+\d^{(k)}) J_{[i]}^{(k)}+\d_{[i]}^{(k)}+(\d^{(k)}-\d_{[i]}^{(k)}) \e_{[i]}^{(k)},
\end{eqnarray*}
where we set $\d^{(k)}:=\max_i \d_{[i]}^{(k)}$. Summing up in $i$, we obtain
\begin{equation} \label{cfjf}
\cJ \circ \bFS^{(k)} \leq (1+\d^{(k)}) \cJ^{(k)}+\tau^{(k)}
\end{equation}
with $\tau^{(k)} \to 0$. Thus
\begin{eqnarray*}
\cD^{(k)}(\bH) &=& \cD^{(k)}(f^\ast \bH) \\
&=& \cJ^{(k)}(f^\ast \bH)+(\cL-\hat{\cL})(\bFS^{(k)}(f^\ast \bH)) \\
&\geq& \bigg[ \frac{1}{1+\d^{(k)}}-(1-\d) \bigg] \cJ(\bFS^{(k)}(f^\ast H))-\frac{\tau^{(k)}}{1+\d^{(k)}}-C.
\end{eqnarray*}
Hence if we take $k$ sufficiently large so that $(1+\d^{(k)})^{-1}>1-\d$, we obtain
\begin{equation} \label{ebm}
\cD^{(k)}(\bH) \geq \d' \cJ_G^{(k)}(\bH)-C'
\end{equation}
for some ($k$-independent) uniform constants $\d', C'>0$. We take a sequence $\{\bH_\ell\} \subset \prod_i \cH_{[i]}^{(k)}$ so that $\lim_{\ell \to \infty} \cD^{(k)}(\bH_\ell)=\inf_{\prod_i \cH_{[i]}^{(k)}} \cD^{(k)}$. Since $\cD^{(k)}$ is $G$-invariant, by multiplying some elements in $G$, we may assume that $\{ \bH_\ell\}$ is contained in a sublevel set of $\cJ^{(k)}$ which is compact from \cite[Lemma 7.6]{BBGZ13}. Hence by passing to a subsequence, $\bH_\ell$ converges to a minimizer $\bH^{(k)}$ of $\cD^{(k)}$, and hence balanced.
Since $\bH^{(k)}$ is a minimizer of $\cD^{(k)}$, for any $\bphi \in \prod_i \cH_{[i]}$, we have
\begin{equation} \label{minqD}
\cD^{(k)}(\bH^{(k)}) \leq \cD^{(k)}(\bHilb_{\bMA({\bf 0})}^{(k)}(\bphi)).
\end{equation}
By using the asymptotic expansion of Bergman kernel (\cf Theorem \ref{asexp}), we have $\bFS^{(k)}(\bHilb_{\bMA({\bf 0})}^{(k)}(\bphi)) \to \bphi$ in the $C^0$-topology, which yields that $\cL(\bFS^{(k)}(\bHilb_{\bMA({\bf 0})}^{(k)}(\bphi)) \to \cL(\bphi)$. Moreover, we also have $\AM_{[i]}^{(k)}(\Hilb_{[i], \bMA({\bf 0})}^{(k)}(\bphi)) \to \AM_{[i]}(\phi_{[i]})$ (\cf \cite[Theorem A]{BB10}). Thus as in the formula \cite[page 242, line 21]{BBGZ13} in the case $N=1$, we know that
\begin{equation} \label{convqD}
\cD^{(k)}(\bHilb_{\bMA({\bf 0})}^{(k)}(\bphi)) \to \cD(\bphi).
\end{equation}
From \eqref{minqD} and \eqref{convqD}, we know that for any $\bphi \in \prod_i \cH_{[i]}$, there exists a sequence $\g^{(k)}=\g^{(k)}(\bphi) \to 0$ such that
\begin{equation} \label{ufb}
\cD^{(k)}(\bH^{(k)})-\cD(\bphi) \leq \g^{(k)}.
\end{equation}
Set $\bphi^{(k)}:=\bFS^{(k)}(\bH^{(k)})$. From \eqref{ebm}, after adjusting by some element in $G$, we may further assume that $\bH^{(k)}$ satisfies $\cD^{(k)}(\bH^{(k)}) \geq \d' \cJ^{(k)}(\bH^{(k)})-C'$. Putting all things together, we can compute
\begin{eqnarray*}
\cD(\bphi^{(k)})-\cD^{(k)}(\bH^{(k)}) &=&\cJ(\bphi^{(k)})-\cJ^{(k)}(\bH^{(k)})\\
&\leq& \d^{(k)} \cJ^{(k)}(\bH^{(k)})+\tau^{(k)}\\
&\leq& \frac{\d^{(k)}}{\d'} \cD^{(k)}(\bH^{(k)})+\tau^{(k)}+\frac{\d^{(k)}C'}{\d'}\\
&\to& 0
\end{eqnarray*}
since $\cD^{(k)}(\bH^{(k)})$ is uniformly controlled from above by \eqref{ufb}. Thus we obtain
\[
\liminf_{k \to \infty} \cD(\bphi^{(k)}) \leq \cD(\bphi)
\]
for all $\bphi \in \prod_i \cH_{[i]}$. In particular, $\bphi^{(k)}$ is a minimizing sequence of $\cD$. Hence, by using the coercivity of $\cD$ and compactness of the sublevel sets of $\cJ$ in the weak topology (\cf \cite[Lemma 3.3]{BBGZ13}), we can extract a weak convergent subsequence of $\bphi^{(k)}$, whose limit $\bphi^{(\infty)}$ should be CKE by the lower semicontinuity for $\cD$. We remark that the above arguments hold for {\it any} subsequence of $\bH_\ell$. Combining with the uniqueness result of CKE metrics modulo the $G$-action, we conclude that $\bphi^{(k)}$ converges to $\bphi^{(\infty)}$ weakly modulo the $G$-action without taking a subsequence. This completes the proof.
\end{proof}
%==============Section 5==========================
\section{Towards the $C^\infty$-convergence} \label{TCR}
%==============Subsection 5.1==========================
\subsection{Bergman kernel asymptotics and almost balanced metrics} \label{BKABM}
In what follows, we always assume that $\lambda=-1$, or $\lambda=1$ and $(X, (L_{[i]}))$ admits a CKE metric with discrete automorphisms.
Let us consider a single polarization $L \to X$ and a smooth probability measure $\nu$. Then for any $\phi \in \cH$, we define the {\it Bergman function} $B_\nu^{(k)}(\phi)$ by
\begin{equation} \label{Bergman}
B_\nu^{(k)}(\phi):=\sum_\a |s_\a|_{k\phi}^2,
\end{equation}
where $(s_\a)$ is {\it any} ONB with respect to $\Hilb_{\nu}^{(k)}(\phi)$. Indeed, the above definition is independent of the choice of $(s_\a)$. The function $B_\nu^{(k)}$ measures the deviation of $S_\nu^{(k)}$ from the identity
\[
S_\nu^{(k)}(\phi)-\phi=\frac{1}{k} \log \bigg( \frac{1}{D^{(k)}} B_\nu^{(k)}(\phi) \bigg).
\]
The asymptotic expansion of Bergman functions has been well studied \cite{Bou90, Cat99, Tia90, Zel98}. Especially, we use the asymptotic expansion of the Bergman function $B_\nu^{(k)}(\phi)$ associated to the space of global sections of $\C \otimes L^{\otimes k}$, where $\C$ denotes the trivial bundle equipped with the fiber metric $\nu/\MA(\phi)$ (\cf \cite{Zel98, MM07}):
\begin{thm} \label{asexp}
We have the following asymptotic expansion of $B_\nu^{(k)}$:
\begin{equation} \label{asebk}
B_\nu^{(k)}(\phi)=(b_0k^n+b_1k^{n-1}+b_2 k^{n-2}+\cdots) \cdot \frac{\MA(\phi)}{\nu},
\end{equation}
where $b_0=L^n/n!$ and each coefficient $b_j$ can be expressed as a polynomial in the Riemannian curvature ${\rm Riem}(\omega_\phi)$, the curvature of the fiber metric $\nu/\MA(\phi)$, their derivatives and contractions with respect to $\omega_\phi$. The above expansion is uniform in $\phi$ and $\nu$ as long as they stay in a compact set of $\phi$ in the $C^\infty$-topology for which $\omega_\phi$ is bounded from below (by some given K\"ahler metric). More precisely, for any integer $p$ and $r$, there exists a constant $C_{p,r}$ such that
\[
\bigg\| \frac{\nu}{\MA(\phi)}B_\nu^{(k)}(\phi) -\sum_{j=0}^p b_j k^{n-j} \bigg\|_{C^r} \leq C_{p, r} \cdot k^{n-p-1}.
\]
\end{thm}
For the coupled settings, let $\bnu$ be an $N$-tuple of probability measures, and define the function $B_{[i],\bnu}^{(k)}(\bphi)$ by the sum of the norm square of any $\Hilb_{[i], \bnu}^{(k)}(\bphi)$-ONB measured by the fiber metric $| \cdot |_{k\phi_{[i]}}^2$. We set
\[
\bar{B}_{[i],\bnu}^{(k)}(\bphi):=\frac{1}{D_{[i]}^{(k)}} B_{[i],\bnu}^{(k)}(\bphi).
\]
Especially, we are interested in the $N$-tuple of Bergman functions $B_{[1], \bmu}^{(k)}(\bphi),\ldots, B_{[N], \bmu}^{(k)}(\bphi)$. By \eqref{Bergman}, a metric $\bphi \in \prod_i \cH_{[i]}$ is balanced if and only if
\[
\bar{B}_{[1],\bmu}^{(k)}(\bphi)=\ldots=\bar{B}_{[N],\bmu}^{(k)}(\bphi)=1.
\]
By Theorem \eqref{asexp}, at the CKE potential $\bphi_{\CKE}$, we obtain
\[
\bar{B}_{[i], \bmu}^{(k)}(\bphi_{\CKE})=1+O(k^{-1}), \quad i=1,\ldots, N.
\]
However it is not enough for proving the $C^\infty$-convergence of balanced metrics. Following argument is essentially same as \cite[Section 4.1]{Don01}: for any integer $\ell$, we will find a potential $\bphi_\ell=(\phi_{[i], \ell})$ such that
\[
\bar{B}_{[i], \bmu}^{(k)}(\bphi_\ell)=1+O(k^{-\ell-1}), \quad i=1,\ldots, N.
\]
We call such metrics {\it almost balanced metrics}. More precisely, we try to construct such a metric by adding the polynomial of $k^{-1}$ to $\bphi_{\CKE}$ whose the coefficients $\eta_{[i], j}$'s are smooth functions:
\[
\phi_{[i], \ell}:=\phi_{\CKE, [i]}+\sum_{j=1}^\ell k^{-j} \eta_{[i], j},
\]
We note that for any fixed $\ell$, the potential $\phi_{[i], \ell}$ stays in a compact set in the $C^\infty$-topology since it includes only finitely many perturbations. Now we will find $\bet_1:=(\eta_{[1],1},\ldots, \eta_{[N], 1})$  satisfying the above property. For any $\bphi \in \prod_i \cH_{[i]}$, we write the coefficients in the asymptotic expansion of $\bar{B}_{[i], \bmu}^{(k)}(\bphi)$ as
\[
B_{[i], \bmu}^{(k)}(\bphi)=(b_{[i], 0}(\bphi) k^n+b_{[i], 1}(\bphi) k^{n-1}+b_{[i], 2}(\bphi) k^{n-2}+\cdots) \cdot \frac{\MA(\phi_{[i]})}{\mu(\bphi)}.
\]
Since each $b_{[i],p}(\bphi)$ are polynomials in the Riemannian curvature ${\rm Riem}(\omega_{\phi_{[i]}})$, $\mu(\bphi)/\MA(\phi_{[i]})$, together with their derivatives and metric contractions, it follows that
\[
b_{[i],p}(\bphi_1)=b_{[i],p}(\bphi_{\CKE})+O(k^{-1}), \quad p=0,1,2,\ldots.
\]
On the other hand, from Proposition \ref{vfmu}, the linearization of $\MA(\phi_{[i]})/\mu(\bphi)$ in the direction $(\d \phi_{[1]},\ldots, \d \phi_{[N]})$ is given by
\[
\frac{\MA(\phi_{[i]})}{\mu(\bphi)}\bigg(\Delta_{\phi_{[i]}} \d \phi_{[i]}+\lambda \sum_j \d \phi_{[j]} -\lambda \sum_j \int_X \d \phi_{[j]} \mu(\bphi) \bigg).
\]
Thus we have an expansion
\begin{eqnarray*}
B_{[i], \bmu}^{(k)}(\bphi_1) &=& \frac{L_{[i]}^n}{n!} k^n+\bigg( b_{[i],1}(\bphi_{\CKE})+\Delta_{\phi_{\CKE, [i]}} \eta_{[i],1}+\lambda \sum_j \eta_{[j], 1} -\lambda \sum_j \int_X \eta_{[j], 1} \mu(\bphi_{\CKE}) \bigg) k^{n-1}\\
&+& O(k^{n-2}).
\end{eqnarray*}
Hence what we have to do is to find an $N$-tuple of functions $\bet_1$ solving the coupled Poisson equations:
\begin{equation} \label{cfeq}
\begin{cases}
\Delta_{\phi_{\CKE, [1]}} \eta_{[1],1}+\lambda \sum_j \eta_{[j],1}-\lambda \sum_j \int_X \eta_{[j],1} \mu(\bphi_{\CKE}) +b_{[1],1}(\bomega_{\bphi_{\CKE}})=c_{[1],1}\\
\hspace{20mm}
\vdots \\
\Delta_{\phi_{\CKE, [N]}} \eta_{[N],1}+\lambda \sum_j \eta_{[j],1}-\lambda \sum_j \int_X \eta_{[j],1} \mu(\bphi_{\CKE}) +b_{[N],1}(\bomega_{\bphi_{\CKE}})=c_{[N],1}\\
\end{cases}
\end{equation}
with some constants $c_{[1],1}, \ldots, c_{[N],1}$. If we can do this, we have $\bar{B}_{[i], \bmu}^{(k)}(\bphi_1)=1+O(k^{-2})$ for all $i=1,\ldots, N$ with this $\bet_1$, and will continue the argument. For the general $\ell$th-step, taking into the account that $b_{[i],p}(\bphi_\ell)=b_{[i],p}(\bphi_{\ell-1})+O(k^{-\ell})$, the first contribution of $\bet_\ell$ to $b_{[i],p}(\bphi_\ell)$ occurs at $O(k^{-\ell})$. Thus the similar argument shows that the equation that we should solve is given by
\begin{equation} \label{cleq}
\begin{cases}
\Delta_{\phi_{\CKE, [1]}} \eta_{[1],\ell}+\lambda \sum_j \eta_{[j], \ell}-\lambda \sum_j \int_X \eta_{[j],\ell} \mu(\bphi_{\CKE})+F_{[1],\ell}(\bomega_{\bphi_{\CKE}}, \bet_1, \ldots, \bet_{\ell-1})=c_{[1],\ell} \\
\hspace{20mm}
\vdots \\
\Delta_{\phi_{\CKE, [N]}} \eta_{[N], \ell}+\lambda \sum_j \eta_{[j], \ell}-\lambda \sum_j \int_X \eta_{[j],\ell} \mu(\bphi_{\CKE})+F_{[N],\ell}(\bomega_{\bphi_{\CKE}}, \bet_1, \ldots, \bet_{\ell-1})=c_{[N],\ell}\\
\end{cases}
\end{equation} 
where $c_{[1],\ell}, \ldots, c_{[N],\ell}$ are some constants and the functions $F_{[i],\ell} \in C^\infty(X;\R)$  depend smoothly on the data $(\bomega_{\bphi_{\CKE}}, \bet_1, \ldots, \bet_{\ell-1})$ constructed in the previous steps. From Proposition \ref{wlaplacian}, we may take the constants
\[
c_{[i],\ell}=\int_X F_{[i],\ell} \mu(\bphi_{\CKE}), \quad i=1,\ldots, N,
\]
so that $(c_{[i],\ell}-F_{[i],\ell})_{i=1,\ldots,N}$ lies in the orthogonal complement of $\Ker \cP_{\bphi_{\CKE}}$ with respect to the Hermitian product \eqref{lproduct}, which is equivalent to say that the system of equations \eqref{cleq} has a solution $\bet_\ell$. We note that the solution $\bet_\ell$ is real since the CKE condition yields that the operator $\cP_{\bphi_{\CKE}}$ is real. Summarizing up, we obtain:
\begin{thm}[Theorem \ref{happroximation}]
Let $X$ be a compact K\"ahler manifold with $\lambda c_1(X)>0$ ($\lambda=\pm 1$), and $(L_{[i]})$ a rational decomposition of $K_X^{\otimes -\lambda}$. In the $\lambda=1$ case, we further assume that the decomposition $(L_{[i]})$ admits a coupled K\"ahler-Einstein metric and $\Aut_0(X)$ is trivial. Let $\bphi_{\CKE} \in \prod_i \cH_{[i]}$ denote the unique coupled K\"ahler-Einstein metric. Then there exist $\bet_1, \ldots, \bet_\ell \in (C^\infty(X;\R))^N$ such that $\bphi_\ell:=\bphi_{\CKE}+\sum_{j=1}^\ell k^{-j} \bet_j$ satisfies
\begin{equation} \label{hberg}
\bar{B}_{[i], \bmu}^{(k)}(\bphi_\ell)=1+O(k^{-\ell-1}), \quad i=1,\ldots, N.
\end{equation}
\end{thm}
%==============Subsection 5.2==========================
\subsection{$R$-bounded geometry}
Take an arbitrary number $R>0$ and consider the single polarization $L \to X$. In this subsection, we review the $R$-bounded geometry and some related results about it (\cf \cite{Don01, Fin10}). We note that in Donaldson's paper \cite{Don01}, he considered ``large'' metrics which live in $2\pi kc_1(L)$, while we consider the rescaled ones in $2\pi c_1(L)$ in order to unify the notations throughout the paper. Also we note that the distance $\dist^{(k)}$ and length of paths in $\cH^{(k)}$ are computed with respect to the normalized Riemannian structure $\langle A, B \rangle_H=\frac{1}{k^2 D^{(k)}} \Tr(A \circ B)$.
\begin{dfn}
We say that a K\"ahler metric $\omega \in 2 \pi c_1(L)$ has $R$-bounded geometry in $C^r$ (with respect to $\hat{\omega}$) if $\omega \geq R^{-1} \hat{\omega}$ and
\[
\|k\omega-k\hat{\omega}\|_{C^r(k\hat{\omega})} <R.
\]
\end{dfn}
\begin{rem} \label{unifk}
The condition $\|\omega-\hat{\omega}\|_{C^0(\hat{\omega})} <R$ yields that $\Tr_{\hat{\omega}} \omega \leq c$. Thus we have another bound $\omega \leq c \cdot \hat{\omega}$ and hence $\osc_X \phi \leq c$ by the Ko\l odziej's $C^0$-estimate \cite{Kol03, Kol05}, where we set $\omega=\hat{\omega}+\dd \phi$.
\end{rem}
\begin{lem}[Lemma 13, \cite{Fin10}] \label{cKL}
Let $H(s)$ be a smooth path in $\cH^{(k)}$. If $\omega(s):=\omega_{\FS^{(k)}(H(s))}$ have $R$-bounded geometry in $C^r$ and $\| \overline{M}_{\MA(\FS^{(k)}(H(s)))}(H(s)) \|_{\op}=K \cdot k^{-n}$, then
\[
\|k\omega(0)-k\omega(1)\|_{C^{r-2}(k\hat{\omega})}<cKL \cdot k^{\frac{3}{2}n}.
\]
where $c$ is a uniform constant which only depends on $r$ and $R$, and $L$ is a length of the path $\{H(s)\}_{0 \leq s \leq 1}$.
\end{lem}
\begin{lem}[Lemma 14, \cite{Fin10}] \label{bnhd}
Let $H^{(k)} \in \cH^{(k)}$ be a sequence of Hermitian forms such that the corresponding sequence of K\"ahler forms $\omega_{\FS^{(k)}(H^{(k)})}$ has $R/2$-bounded geometry in $C^{r+2}$ and such that $\|\overline{M}_{\MA(\FS^{(k)}(H^{(k)}))}(H^{(k)}) \|_{\op}=O(k^{-n})$. Then there exists a constant $c>0$ (which depends only on $R$, but not on $k$) such that if $H \in \cH^{(k)}$ satisfies $\dist^{(k)}(H^{(k)}, H)<c k^{-\frac{n}{2}-1}$, then the corresponding K\"ahler form $\omega_{\FS^{(k)}(H)}$ has $R$-bounded geometry in $C^r$.
\end{lem}
\begin{lem}[Proposition 24, \cite{Fin10}] \label{opop}
For any $H, H^\dag \in \cH^{(k)}$, we have
\[
\|\overline{M}_{\MA(\FS^{(k)}(H^\dag))}(H^\dag)\|_{\op} \leq \exp(c k^{\frac{n}{2}+1} \dist^{(k)}(H, H^\dag)) \|\overline{M}_{\MA(\FS^{(k)}(H))}(H)\|_{\op}
\]
for some ($k$-independent) uniform constant $c>0$.
\end{lem}
We also use the following lemma obtained in \cite[Section 5]{PS04}:
\begin{lem} \label{projest}
Assume that $G$ is trivial and for $H \in \cH^{(k)}$, $\FS^{(k)}(H)$ has $R$-bounded geometry in $C^r$ ($r \geq 4$). Then for any $A \in \sqrt{-1} \fu_H^{(k)}$, we have the following inequalities:
\begin{equation} \label{projo}
k \|A\|_{H}^2 \leq c \| \xi_A \|_{L^2(\MA({\FS^{(k)}(H)}))}^2,
\end{equation}
\begin{equation} \label{projs}
\| \xi_A \|_{L^2(\MA({\FS^{(k)}(H)}))}^2=\| \xi_A^\top \|_{L^2(\MA({\FS^{(k)}(H)}))}^2+\| \xi_A^\bot \|_{L^2(\MA({\FS^{(k)}(H)}))}^2,
\end{equation}
\begin{equation} \label{projt}
c' \| \xi_A^\top \|_{L^2(\MA({\FS^{(k)}(H)}))}^2 \leq k \| \xi_A^\bot \|_{L^2(\MA({\FS^{(k)}(H)}))}^2,
\end{equation}
where $c, c'>0$ are constants which depend only on $R$, and $\| \cdot \|_{L^2(\MA({\FS^{(k)}(H)}))}^2$ denotes the $L^2$ norm with respect to the probability measure $\MA({\FS^{(k)}(H)})$ on the base and the Fubini-Study metric $| \cdot |_{\aFS(H)}^2$ on the fiber. Also the sections $\xi_{A_{[i]}}^\top$ and $\xi_{A_{[i]}}^\bot$ are the components of the pointwise decomposition defined in \eqref{pwdcomp}.
\end{lem}
%==============Subsection 5.3==========================
\subsection{The balancing flow}
We consider the gradient flow of $\cD^{(k)}$ referred as the {\it balancing flow}:
\begin{equation} \label{balflowr}
\begin{cases}
\ddt H_{[1]}(t)=k(\Id-D_{[1]}^{(k)} \overline{M}_{[1], \bmu}(\bH(t))) \\
\hspace{20mm} \vdots \\
\ddt H_{[N]}(t)=k(\Id-D_{[N]}^{(k)} \overline{M}_{[N], \bmu}(\bH(t))).
\end{cases}
\end{equation}
This is an ODE on a finite dimensional symmetric space and short time existence follows from the standard ODE theory. Set
\[
\cR^{(k)}(\bH):=k^2 \big\| \big(\Id-D_{[i]}^{(k)} \overline{M}_{[i], \bmu}(\bH) \big) \big\|_{\bH}^2.
\]
In order to find the fixed point of the flow, we need the following lemma:
\begin{lem} \label{regularityofbf}
Assume that there exist constants $C, \d>0$ such that the balancing flow $\bH(t)$ satisfies the following conditions:
\begin{enumerate}
\item $\cR^{(k)}(\bH(0)) \leq \frac{1}{16} \d^2 C^2$.
\item We have
\[
\ddt \cR^{(k)}(\bH(t)) \leq -\d \cR^{(k)}(\bH(t))
\]
as long as $\bH(t) \in B_{\bH(0)}(C)$.
\end{enumerate}
(Here the constants $C, \d$ are allowed to depend on $k$.) Then the balancing flow starting from $\bH(0)$ exists for all time and converges to a balanced metric $\bH^{(k)}$ which lies in $B_{\bH(0)}(2\cR^{(k)}(\bH(0))^{1/2}/\d)$.
\end{lem}
\begin{proof}
Let $T<\infty$ be the maximum time such that the flow $\bH(t)$ exists and $\bH(t) \in B_{\bH(0)}(C)$ for all $t \in [0,T)$. Then we have
\[
\ddt \cR^{(k)}(\bH(t)) \leq -\d \cR^{(k)}(\bH(t)), \quad t \in [0,T].
\]
Thus we have
\begin{equation} \label{edr}
\cR^{(k)}(\bH(t)) \leq e^{-\d t} \cR^{(k)}(\bH(0)), \quad t \in [0,T].
\end{equation}
Integrating on $[0,T]$, we observe that
\[
\dist^{(k)}(\bH(0), \bH(T)) \leq \int_0^T \cR^{(k)}(\bH(t))^{\frac{1}{2}} dt \leq \frac{2}{\d} \cR^{(k)}(\bH(0))^{\frac{1}{2}} \leq \frac{C}{2}.
\]
As argued in \cite[Section 5.3]{Has15}), we take a cut-off function which is equal to $1$ on $B_{\bH(0)}(C)$ and supported on $B_{\bH(0)}(2C)$. Then multiplying the RHS of \eqref{balflowr} does not change the flow as long as it stays in $B_{\bH(0)}(C)$. So by the completeness, there exists $T^\ast>T$ such that the flow $\bH(t)$ exists and $\bH(t) \in B_{\bH(0)}(C)$ for $t \in [0, T^\ast)$, which contradicts the assumption for $T$. Eventually we have $T=\infty$, and by passing to a subsequence the balancing flow $\bH(t)$ converges to some element $\bH^{(k)} \in B_{\bH(0)}(2\cR^{(k)}(\bH(0))^{1/2}/\d)$. Moreover, by  letting $t \to \infty$ in \eqref{edr}, we have $\cR^{(k)}(\bH^{(k)})=0$, which shows that $\bH^{(k)}$ is balanced. Since $G$ is trivial                                      , balanced metrics are unique by Proposition \ref{qdf} (3). Thus we conclude that the flow $\{\bH(t)\}$ converges without taking subsequences.
\end{proof}
%==============Subsection 5.4==========================
\subsection{Proof of Theorem \ref{crconv}} \label{prfscb}
In what follows, we concentrate on the balancing flow $\bH=\bH(t)$ starting from the almost balanced metric $\bH(0):=\bHilb_{\bmu}^{(k)}(\bphi_\ell)$.
\begin{lem} \label{estop}
We have
\[
\| \overline{M}_{[i], \bMA}(\bH(0))\|_{\op}=O(k^{-n}).
\]
\end{lem}
\begin{proof}
Let $(s_{[i],\a})_{\a=1,\ldots, D_{[i]}^{(k)}}$ denotes a $\bH(0)$-ONB. Then we can compute the matrix representation of $D_{[i]}^{(k)} \overline{M}_{[i],\bMA}(\bH(0))$ as
\begin{eqnarray*}
\big( D_{[i]}^{(k)} \overline{M}_{[i], \bMA}(\bH(0)) \big)_{\a \b}&=& \int_X \frac{(s_{[i],\a}, s_{[i],\b})_{k \phi_{[i], \ell}}}{\bar{B}_{[i], \bmu}^{(k)}(\bphi_\ell)} \MA(\FS_{[i]}^{(k)}(H_{[i]}(0)))\\
&=& \int_X \frac{(s_{[i],\a}, s_{[i],\b})_{k \phi_{[i], \ell}}}{\bar{B}_{[i], \bmu}^{(k)}(\bphi_\ell)}(1+O(k^{-\ell-2})) \MA(\phi_{[i],\ell}) \\
&=& \int_X \frac{(s_{[i],\a}, s_{[i],\b})_{k \phi_{[i], \ell}}}{1+O(k^{-1})}(1+O(k^{-\ell-2})) \mu(\bphi_\ell) \\
&=& \int_X (s_{[i],\a}, s_{[i],\b})_{k \phi_{[i], \ell}}(1+O(k^{-1})) \mu(\bphi_\ell),
\end{eqnarray*}
where we used the uniform asymptotic expansion \eqref{asebk} and the fact that $\omega_{\FS_{[i]}^{(k)}(H_{[i]}(0))}=\omega_{\phi_{[i], \ell}}+O(k^{-\ell-2})$. Thus we have
\[
\big( D_{[i]}^{(k)} \overline{M}_{[i], \bMA}(\bH(0))-\Id \big)_{\a \b}= \int_X (s_{[i],\a}, s_{[i],\b})_{k \phi_{[i], \ell}}F^{(k)} \mu(\bphi_\ell)
\]
for some function $F^{(k)} \in C^{\infty}(X;\R)$ with $F^{(k)}=O(k^{-1})$. On the other hand, we can decompose $D_{[i]}^{(k)} \overline{M}_{[i], \bMA}(\bH(0))-\Id$ as
\[
D_{[i]}^{(k)} \overline{M}_{[i], \bMA}(\bH(0))-\Id=p_{[i]} \circ m_{F^{(k)}} \circ \iota_{[i]},
\]
where $\iota_{[i]} \colon H^0(X,L_{[i]}^{\otimes k}) \longrightarrow L^2 (X,L_{[i]}^{\otimes k})$ is the inclusion into the space of all $L^2$-integrable sections, $m_{F^{(k)}}$ is multiplication by $F^{(k)}$, $p_{[i]} \colon L^2(X,L_{[i]}^{\otimes k}) \longrightarrow H^0(X,L_{[i]}^{\otimes k})$ is the $L^2$-orthogonal projection, and all of these $L^2$ items are defined with respect to the $L^2$-inner product $\int_X (\cdot, \cdot)_{k \phi_{[i],\ell}} \mu(\bphi_\ell)$. Hence we obtain
\[
\|D_{[i]}^{(k)} \overline{M}_{[i], \bMA}(\bH(0))-\Id \|_{\rm op} \leq \|F^{(k)}\|_{C^0}=O(k^{-1}).
\]
Using the fact that $D_{[i]}^{(k)}=O(k^n)$, we obtain the desired result.
\end{proof}
\begin{lem}
We have
\[
\cR^{(k)}(\bH(0))=O(k^{-2\ell-4}).
\]
\end{lem}
\begin{proof}
Let $(s_{[i],\a})$ be an $H_{[i]}(0)$-orthonormal and $T_{[i], \bmu}^{(k)}(\bH(0))$-orthogonal basis. Then
\[
\Id-D_{[i]}^{(k)} \overline{M}_{[i], \bmu}(\bH(0))=\diag(1-\|s_{[i],1}\|_{T_{[i], \bmu}^{(k)}(\bH(0))}^2,\ldots, 1-\|s_{[i], D_{[i]}^{(k)}} \|_{T_{[i], \bmu}^{(k)}(\bH(0))}^2),
\]
and each component is computed as
\begin{eqnarray*}
1-\|s_{[i],\a}\|_{T_{[i],\bmu}^{(k)}(H(0))}^2&=&\int_X \bigg[ 1-\frac{|s_{[i],\a}|_{k\phi_{[i], \ell}}^2}{\bar{B}_{[i], \bmu}^{(k)}(\bphi_\ell)} \bigg] \mu(\tilde{S}_{\bmu}^{(k)}(\bphi_\ell)) \\
&=& \int_X \bigg[ 1-\frac{|s_{[i],\a}|_{k\phi_{[i], \ell}}^2}{1+O(k^{-\ell-1})} \bigg] (1+O(k^{-\ell-2}))\mu(\bphi_\ell) \\
&=& O(k^{-\ell-1}),
\end{eqnarray*}
where we used $S_{[i], \bmu}^{(k)}(\bphi_\ell)=\phi_{[i],\ell}+O(k^{-\ell-2})$ and $\int_X |s_{[i],\a}|_{k\phi_{[i],\ell}}^2 \mu(\bphi_\ell)=1$. Thus we obtain the desired statement.
\end{proof}
For any sufficiently large integer $r$, we know that the sequence $\{\bomega_{\bphi_\ell} \}_{k=1,2,\ldots}$ has $R/2$-bounded geometry in $C^r$ for some $R>0$ from the construction. We seek constants $\d, C>0$ satisfying the properties in Lemma \ref{regularityofbf}:
\begin{lem} \label{cobf}
Assume $\ell>\frac{n}{2}$. Then for sufficiently large $k$, the balancing flow $\bH(t)$ starting from $\bH(0)=\bHilb_{\bmu}^{(k)}(\bphi_\ell)$ converges to a balanced metric $\bH^{(k)}$ with $\dist^{(k)}(\bH(0), \bH^{(k)})=O(k^{-\ell-1})$.
\end{lem}
\begin{proof}
In what follows, let $c_j>0$ ($j=1,2,\ldots$) be uniform constants depending only on $R$ (but not on $k$). By Lemma \ref{bnhd} and Lemma \ref{estop}, we also know that there exists a uniform constant $c_1>0$ such that for any $\bH(t) \in B_{\bH(0)}(c_1k^{-\frac{n}{2}-1})$, the corresponding K\"ahler forms $\bomega_{\bFS^{(k)}(\bH(t))}$ has $R$-bounded geometry in $C^{r-2}$. Now we apply the Hessian formula for $\cD^{(k)}$ (\cf Proposition \ref{hessfqd}):
\begin{eqnarray*}
\ddt \cR^{(k)}(\bH) &=& -\frac{d^2}{dt^2}\cD^{(k)}(\bH)\\
&=& -2\nabla^2 \cD^{(k)}|_{\bH}(\dot{\bH},\dot{\bH}) \\
&=& -2k^{-1} \sum_i \int_X |\xi_{\dot{H}_{[i]}}^\bot|_{\aFS(H_{[i]})}^2 \mu(\bFS^{(k)}(\bH))\\
&+& 2k^{-2}(\cP_{\bFS^{(k)}(\bH)} (h_{\dot{H}_{[i]}}), (h_{\dot{H}_{[i]}}))_{\bmu(\bFS^{(k)}(\bH))},
\end{eqnarray*}
where we used the fact that $\bH$ is the gradient flow of $\cD^{(k)}$ in the first and second equality. We note that in general, if $(M,g)$ is a Riemannian manifold and $\g(t)$ is the gradient flow of $F \in C^\infty(X;\R)$, a direct tensor computation shows that
\[
\frac{d^2}{dt^2} F(\g(t))=-2g(\nabla_{\dot{\g}} \nabla F, \nabla F)=2\nabla F(\nabla_{\dot{\g}} \dot{\g})=2(\nabla^2 F)(\dot{\g},\dot{\g}).
\]
The operator $\cP_{\bFS^{(k)}(\bH)}$ is non-positive by Proposition \ref{wlaplacian}. On the other hand, applying Remark \ref{unifk} and Lemma \ref{projest} \footnote{We need the triviality of $G$ to use \eqref{projt}.}, we have
\begin{eqnarray*}
\int_X |\xi_{\dot{H}_{[i]}}^\bot|_{\aFS(H_{[i]})}^2 \mu(\bFS^{(k)}(\bH)) &\geq& c_2 \int_X |\xi_{\dot{H}_{[i]}}^\bot|_{\aFS(H_{[i]})}^2 \mu({\bf 0})\\
&\geq& c_3 \int_X |\xi_{\dot{H}_{[i]}}^\bot|_{\aFS(H_{[i]})}^2 \MA(0_{[i]})\\
&\geq& c_4 \int_X |\xi_{\dot{H}_{[i]}}^\bot|_{\aFS(H_{[i]})}^2 \MA(\FS_{[i]}^{(k)}(H_{[i]}))\\
&\geq& c_5 \|\dot{H}_{[i]}\|_{H_{[i]}}^2
\end{eqnarray*}
for any $\bH(t) \in B_{\bH(0)}(c_1k^{-\frac{n}{2}-1})$. Hence we can find a uniform constant $c_6>0$ such that
\[
\ddt \cR^{(k)}(\bH) \leq -c_6 k^{-1} \cR^{(k)}(\bH)
\]
for $\bH(t) \in B_{\bH(0)}(c_1k^{-\frac{n}{2}-1})$. Since $\cR^{(k)}(\bH(0))=O(k^{-2\ell-4})$, by taking $\ell$ so that $\ell>\frac{n}{2}$, we have
\[
\cR^{(k)}(\bH(0)) \leq \frac{1}{16} (\d')^2 (C')^2
\]
for sufficiently large $k$ with $C':=c_1k^{-\frac{n}{2}-1}$ and $\d':=c_6 k^{-1}$. Thus by Lemma \ref{regularityofbf}, we obtain convergence of balancing flow to the balanced metric $\bH^{(k)}=(H_{[i]}^{(k)})$ with an estimate
\[
\dist^{(k)}(\bH(0), \bH^{(k)})=O(k^{-\ell-1}).
\]
\end{proof}
Now we are ready to prove Theorem \ref{crconv}.
\begin{proof}[Proof of Theorem \ref{crconv}]
From Lemma \ref{cobf} and Lemma \ref{bnhd}, we find that $\omega_{\FS_{[i]}^{(k)}(H_{[i]}^{(k)})}$ has $R$-bounded geometry in $C^{r-2}$. Combining with Lemma \ref{opop}, we also know that $\|\overline{M}_{[i],\bMA}(\bH^{(k)})\|_{\op}=O(k^{-n})$ as long as $\ell>\frac{n}{2}$. Then it follows from Lemma \ref{cKL} that after scaling metrics, we have
\[
\bigg\|\omega_{\FS_{[i]}^{(k)}(H_{[i]}(0))}-\omega_{\FS_{[i]}^{(k)}(H_{[i]}^{(k)})} \bigg\|_{C^{r-4}}=O(k^{\frac{3}{2}n+\frac{r}{2}-\ell-2}).
\]
On the other hand, by using the (uniform) asymptotic expansion of Bergman kernel (\cf Theorem \ref{asexp}), we know that $\bomega_{\bFS^{(k)}(\bH(0))}$ converges to the CKE metric $\bomega_{\CKE}$ in the $C^\infty$-topology. Therefore if $\ell>\frac{3}{2}n+\frac{r}{2}-2$, then we obtain
\[
\bomega_{\bFS^{(k)}(\bH^{(k)})} \to \bomega_{\CKE}
\]
in the $C^{r-4}$-topology. Finally, the balanced metric $\bH^{(k)}$ obtained as the limit point of the flow is independent of $\ell$ chosen to define $\bH(0)$. Then as in \cite[page 511]{Don01}, we can obtain the $C^\infty$-convergence by the uniqueness of balanced metrics since $G$ is trivial (\cf Proposition \ref{qdf} (3)). Indeed, the above argument shows that for any integer $r$, there exists $k_r$ such that the balanced metric $\bomega_{\bFS^{(k)}(\bH^{(k)})}$ exists for all $k \geq k_r$ and $\bomega_{\bFS^{(k)}(\bH^{(k)})}$ converges to $\bomega_{\CKE}$ in $C^r$. We write this sequence as $\{\bomega_{k,r}\}_{k \geq k_r}$ Then we use the uniqueness result to get $\bomega_{k,r}=\bomega_{k,r+j}$ for all $j \in \N$ and $k \geq \max\{k_r,k_{r+j}\}$, which yields the $C^\infty$-convergence. This completes the proof.
\end{proof}
%==============References==========================

\end{document}